\documentclass[12pt]{article}

\usepackage[top=2.5cm, bottom=3cm, left=3cm, right=3cm]{geometry} 

\usepackage{amsmath, amssymb, amscd, amsthm}
\usepackage{mathrsfs} 

\usepackage{graphicx}
\usepackage{float}
\usepackage{tikz}
\usetikzlibrary{cd, intersections, calc, arrows.meta}

\usepackage[title]{appendix} 
\usepackage[hidelinks]{hyperref}

\newtheorem{theorem}{Theorem}[section]
\newtheorem{lemma}[theorem]{Lemma}
\newtheorem{proposition}[theorem]{Proposition}
\newtheorem{cor}[theorem]{Corollary}

\theoremstyle{definition}
\newtheorem{definition}[theorem]{Definition}
\newtheorem{example}[theorem]{Example}
\newtheorem{remark}[theorem]{Remark}

\newtheorem{question}[theorem]{Question}

\numberwithin{equation}{section} 

\newcommand{\B}{\mathbb{B}}
\newcommand{\C}{\mathbb{C}}

\newcommand{\F}{\mathbb{F}}
\newcommand{\K}{\mathbb{K}}
\newcommand{\M}{\mathbb{M}}
\newcommand{\N}{\mathbb{N}}

\newcommand{\R}{\mathbb{R}}

\newcommand{\cI}{\mathcal{I}} \newcommand{\cJ}{\mathcal{J}}

\newcommand{\cO}{\mathcal{O}} \newcommand{\cP}{\mathcal{P}}

\newcommand{\cX}{\mathcal{X}}
 \newcommand{\cZ}{\mathcal{Z}}

\DeclareMathOperator{\diag}{diag}

\DeclareMathOperator{\Lip}{Lip}

\DeclareMathOperator{\Tr}{Tr}

\renewcommand{\Re}{\operatorname{Re}}
\renewcommand{\Im}{\operatorname{Im}}
\newcommand{\inpr}[2]{\langle{#1,#2}\rangle}

\title{A family of matrix flows converging to normal matrices}

\newcommand{\email}[1]{\href{mailto:#1}{\texttt{#1}}}

\author{Masaki Izumi
	\thanks{Supported in part by JSPS KAKENHI Grant Number 25K00912}\\
	Graduate School of Science \\
	Kyoto University \\
	Sakyo-ku, Kyoto 606-8502, Japan \\
	\email{izumi@math.kyoto-u.ac.jp}}

\begin{document} 
\maketitle
\begin{abstract} 
The celebrated Antezana-Pujals-Stojanoff Theorem states that the iterated Aluthge transforms of 
an arbitrary matrix converge to a normal matrix.
We introduce a family of matrix flows that share this convergence property 
by defining them through ordinary differential equations. 
The family includes a continuous analogue of the Aluthge transform, 
as well as a differential equation discussed by Haagerup in the context of II$_1$ factors. 
We also examine the same type of flows in the setting of Hilbert space operators equipped with unitarily invariant norms. 
\end{abstract}

\section{Introduction} 
This work stems from the study of two examples of operator transformations. 
One of them is the Aluthge transform introduced in \cite{A1990}.  
Throughout the paper, we denote the algebra of bounded operators on a Hilbert space $H$ by $\B(H)$. 
For $T\in \B(H)$, let $T=U|T|$ be its polar decomposition. 
Then the Aluthge transform of $T$, denoted by $\Delta(T)$, is defined by $|T|^{1/2}U|T|^{1/2}$.  
Thanks to Antezana-Pujals-Stojanoff's work \cite[Theorem 4.12]{APS2011}, 
the iterated Aluthge transforms of any matrix converge to a normal matrix. 
For a parameter $0<\lambda<1$, one can similarly define the $\lambda$-Aluthge transform
\[\Delta_\lambda(T)=|T|^\lambda U|T|^{1-\lambda},\] 
and the convergence result holds for $\Delta_\lambda$ as well (see \cite[Theorem 6.1]{APS2011}). 
For a left invertible $T$, we have 
\begin{equation*}
\Delta_\lambda(T)=e^{\lambda \log |T|}Te^{-\lambda \log |T|}=T+\lambda [\log |T|,T]+O(\lambda^2),
\end{equation*}
as $\lambda$ tends to 0. 
Thus it is natural to consider the following ODE as a continuous analogue of the Aluthge transform: 
\begin{equation}\label{cAluthge}
	\frac{dX(t)}{dt}=[\log |X(t)|,X(t)],\quad X(0)=T.
\end{equation}
In fact, we can show that $\{\Delta_{t/n}^n(T)\}_{n}$ converges to $X(t)$ in the operator norm 
(see Proposition \ref{Aluthge}). 

The second example comes from an ODE introduced by Haagerup in his series of lectures 
at MSRI in 2001 (see \cite[Section 2]{HS2009} for a published account).  
For an element $T$ in a II$_1$ factor $M$, he considered the following ODE, which we call 
the \textit{Haagerup equation}:   
\begin{equation}\label{HaagerupODE}
\frac{dX(t)}{dt}=[[X(t)^*,X(t)],X(t)],\quad X(0)=T.
\end{equation}
He showed that  a unique global solution exists on $[0,\infty)$, and that it remains in the similarity orbit 
$\cO_M(T)$ of $T$ in $M$. 
Moreover, the flow $\{X(t)\}_{t\geq 0}$ converges to a normal element in the $*$ distribution sense. 
The above ODE itself makes sense in $\B(H)$, and Haagerup's argument works for any finite von Neumann algebras
---in particular, for matrix algebras. 
It turns out that $X(t)$ in the matrix case is precisely the gradient flow for 
the energy function $\Tr([X^*,X]^2)/4$ (see Proposition \ref{gradient}). 
A similarity between the Aluthge transform and the Haagerup equation was 
already observed in \cite{DS2009}.

These two examples can be unified in the following framework. 
Let $\varphi=(\varphi_1,\varphi_2)$ be a pair of continuous increasing functions on a finite closed 
interval $[a,b]\subset [0,\infty)$ such that $\varphi_1+\varphi_2$ is strictly increasing. 
For $T\in \B(H)$ with $\sigma(|T|),\sigma(|T^*|)\subset [a,b]$, we set 
\[\varphi(T)=\varphi_1(|T|)-\varphi_2(|T^*|),\] 
where $\sigma(S)$ denotes the spectrum of an operator $S$. 
The main subject of this paper is the solution of the following ODE: 
\begin{equation}\label{ODE}
\frac{dX(t)}{dt}=[\varphi(X(t)),X(t)],\quad X(0)=T.\end{equation}
Note that $(\varphi_1(x),\varphi_2(x))=(\log x,0)$ in the first example, and 
$(\varphi_1(x),\varphi_2(x))=(x^2,x^2)$ in the second.

In the matrix case, it turns out that the usual Lipschitz condition on $\varphi_1$ and $\varphi_2$ 
assures not only a unique local solution, but a unique global solution on $[0,\infty)$ 
(see Theorem \ref{exi&uni}), and the flow $F^\varphi_t(T):=X(t)$ is well-defined. 
The solution curve $\{F^\varphi_t(T)\}_{t\geq0}$ stays in the similarity orbit $\cO(T)$ of $T$ 
because $[\varphi(X),X]$ defines a vector field on a subset of the manifold $\cO(T)$, and 
$\{F^\varphi_t(T)\}_{t\geq0}$ is its integral curve. 
One of the main purposes of this paper is to show that the flow $\{F^\varphi_t(T)\}_{t\geq 0}$ converges to a normal matrix for 
any matrix $T$ under certain regularity assumptions of $\varphi_1$ and $\varphi_2$---
assumptions satisfied in both of the above examples (see Theorem \ref{exponential}, Theorem \ref{Cn-1}, Theorem \ref{holomorphic}, and 
Theorem \ref{noninvertible} for the precise assumptions).  

To show the convergence results, we have an advantage of our continuous systems over 
the iterated Aluthge transforms: we can separate the dynamics into that of 
upper triangular matrices and that of unitary matrices.  
(Thanks to the unitary invariance of Eq.(\ref{ODE}), we may always assume that the initial value 
$T$ is upper triangular). 
In fact, unlike in \cite{APS2007}, we do not need the sophisticated theory of dynamical systems; 
instead, ordinary Lyapunov analysis together with matrix analysis suffices.

This paper is organized as follows. 
We discuss only matrices up to Section 3, while general bounded operators on Hilbert spaces are treated  
only in Section 4. 
This arrangement is intended to allow readers interested only in the matrix case to follow the paper easily. 

In Section 2, we set up the basic assumptions.  
Following Haagerup's argument, we show that the operator norm of $X(t)$ is decreasing as long as 
the solution exists. 
This, together with the Lipschitz condition, guarantees the existence of a unique global solution. 

In Section 3, we first show exponential convergence under the $C^1$-condition on $\varphi$ for a diagonalizable initial value  
by a simple Lyapunov function argument (Theorem \ref{exponential}).  
Note that in the case of iterated Aluthge transforms, the corresponding result was already established in \cite{APS2007}. 
An intuitive explanation of these two phenomena is straightforward: 
the convergence occurs within the similarity orbit of the initial value if and only if the initial value is diagonalizable. 
To obtain more general convergence results, we require a more refined estimate of the Lyapunov function. 
For this purpose, we use interpolation polynomials in Theorem \ref{Cn-1}, holomorphic functional calculus 
in Theorem \ref{holomorphic}, and a combination of both in Theorem \ref{noninvertible}.  
Each case requires a different regularity assumption.

In Section 4, we discuss general Hilbert space operators equipped with a unitarily invariant norm.  
More precisely, we consider solutions of Eq.(\ref{ODE}) with $X(t)-T$ in a symmetric normed ideal 
under the condition that $[\varphi(T),T]$ belongs to the same ideal. 
We have no convergence theorem in the infinite dimensional case, and the main goal of this section is to prove  
the existence of a unique global solution under an appropriate Lipschitz condition.  
This again follows from the fact that the operator norm of the solution is decreasing, but the Haagerup argument 
no longer applies. 
Instead, we take an alternative approach using the Dini derivative. 
We show both convergent examples and non-convergent ones, and discuss several questions.  

Throughout the paper, for a matrix or a bounded operator $T$, the symbol $\|T\|$ denotes its operator norm. 
The Schatten $p$-norm is denoted by $\|T\|_p$ for $1\leq p\leq \infty$. 
For matrices, we mainly use the Hilbert--Schmidt norm and the operator norm. 
We denote by $\sigma(T)$ the spectrum of $T$, and by $r(T)$ the spectral radius of $T$. 
The similarity orbit of $T$ is denoted by $\cO(T)$ in the matrix case. 
For Hilbert space operators, there is an issue with the choice of 
a subgroup of the group of bounded invertible operators $\B(H)^{-1}$ to define the similarity orbit of $T$, 
and we therefore use a more descriptive notation in Section 4.  

We used Gemini and ChatGPT for text editing and consistency checks only. 
\section{The basic properties of the solutions}
For the basics of matrix analysis, the reader is referred to \cite{B1997},\cite{HP2014}, and \cite{HJ1991}.  

We denote by $\M_n$ the set of $n$-by-$n$ complex matrices, and by $GL(n,\C)$ the group of its invertible elements. 
For $T$, we denote by $\{s_i(T)\}_{i=1}^n$ its singular numbers with $s_1(T)\geq s_2(T)\geq\cdots \geq s_n(T)$.

\begin{definition} For a pair of real valued continuous functions $\varphi=(\varphi_1,\varphi_2)$ defined 
on a finite closed interval $[a,b]\subset [0,\infty)$, we set up the following conditions: 
\begin{itemize}
\item[(C0)] $\varphi_1$ and $\varphi_2$ are increasing, and $\varphi_1+\varphi_2$ 
is strictly increasing.  
\item[(CL)] In addition to (C0), we assume $\varphi_1, \varphi_2\in \Lip[a,b]$, where 
$\Lip[a,b]$ is the set of Lipschitz functions on $[a,b]$. 
\item[(C1)] In addition to (CL), we assume $\varphi_1, \varphi_2\in C^1[a,b]$ and 
$\varphi_1'(x)+\varphi_2'(x)>0$ for all $x\in [a,b]\setminus \{0\}$. 
\end{itemize}
\end{definition}

We set 
\[\M_n[a,b]=\{X\in \M_n;\; \sigma(|X|)\subset [a,b]\},\] 
and $\varphi(X)=\varphi_1(|X|)-\varphi_2(|X^*|)$ for $X\in \M_n[a,b]$. 
Note that $X\in \M_n[a,b]$ implies $\sigma(|X^*|)\subset [a,b]$ for matrices. 

\begin{remark}
The map $\varphi:\M_n[a,b]\to \M_n$ is continuous under the condition (C0), which can be  
shown by polynomial approximations.  
Furthermore, it is Lipschitz continuous under the condition (CL), which follows from 
\[
\left(
\begin{array}{cc}
|T| & 0 \\
0 &|T^*| 
\end{array}
\right)=
\left|
\left(
\begin{array}{cc}
0 &T^*  \\
T &0 
\end{array}
\right)
\right|\]
and \cite[Theorem 3.5.1]{AP2016}. 
\end{remark}

The next lemma follows from the rearrangement inequality. 

\begin{lemma}\label{rearrangement}
Let $A\in \M_n$ be a self-adjoint matrix, and let $f$ be an increasing function on $\sigma(A)$. 
Let $U\in \M_n$ be a unitary matrix. 
Then 
\[\Tr(f(A)(A-UAU^*))\geq 0,\]
and equality holds if and only if $U$ commutes with $f(A)$. 	
\end{lemma}

\begin{proof}
We choose an orthonormal basis $\{v_i\}_{i=1}^n$ of $\C^n$ consisting of eigenvectors of $A$ 
with $Av_i=\lambda_i v_i$, and assume $\lambda_1\geq \lambda_2\geq\cdots \geq \lambda_n$.  
Then 
\begin{align*}
\Tr(f(A)(A-UAU^*)) &=\sum_{i=1}^n f(\lambda_i)(\lambda_i-\inpr{UAU^*v_i}{v_i}) \\
	&=\sum_{i=1}^n f(\lambda_i)(\lambda_i-\sum_{j=1}^n\lambda_j|\inpr{Uv_j}{v_i}|^2).
\end{align*}
Since  $(|\inpr{Uv_j}{v_i}|^2)_{ij}$ is a doubly stochastic matrix, it is a convex combination 
of permutation matrices, and there exists $c:S_n\to [0,1]$ with 
\[\sum_{\sigma\in S_n}c(\sigma)=1,\]
\[|\inpr{Uv_j}{v_i}|^2=\sum_{\sigma\in S_n}c(\sigma)\delta_{i,\sigma(j)},\]
where $S_n$ is the symmetric group (see \cite[Theorem II.2.3]{B1997}). 
Thus 
\[\Tr(f(A)(A-UAU^*))=\sum_{\sigma\in S_n}c(\sigma)
\sum_{j=1}^n \lambda_j(f(\lambda_j)-f(\lambda_{\sigma(j)})).\] 
Letting 
\[\mu_j=\begin{cases}
	\lambda_j-\lambda_{j+1},  & 1\leq j\leq n-1, \\
	\lambda_n, &j=n,
\end{cases}\]
we get 
\[\Tr(f(A)(A-UAU^*))=\sum_{\sigma\in S_n}c(\sigma)\sum_{i=1}^n 
\mu_i\sum_{j=1}^i(f(\lambda_j)-f(\lambda_{\sigma(j)}))\geq 0.\]

Assume that equality holds. 
We claim that if there exists $i$ with $f(\lambda_{\sigma(i)})<f(\lambda_i)$, we get $c(\sigma)=0$. 
Indeed, for such $i$ we choose the largest number $i'$ with $\lambda_i=\lambda_{i'}$. 
Then $\mu_{i'}=\lambda_{i'}-\lambda_{i'+1}>0$, and 
\[\mu_{i'}\sum_{j=1}^{i'}(f(\lambda_j)-f(\lambda_{\sigma(j)}))>0,\]
which implies $c(\sigma)=0$. 
The claim shows that $c(\sigma)>0$ necessitates $f(\lambda_{\sigma(j)})=f(\lambda_j)$ for all $j$, 
implying that $U$ commutes with $f(A)$. 
\end{proof}

\begin{lemma}\label{positive semi-definite}
Let $\varphi=(\varphi_1,\varphi_2)$ be a pair of functions satisfying the condition (C0), 
and assume $T\in \M_n[a,b]$. 
Then 
\[\Tr(\varphi(T)[T^*,T])\geq 0,\]
and equality holds if and only if $T$ is normal. 
In particular $[\varphi(T),T]=0$ holds if and only if $T$ is normal. 
\end{lemma}

\begin{proof} Let $T=U|T|$ be the polar decomposition. 
We may and do assume that $U$ is a unitary by adding a partial isometry to $U$ if necessary. 
Then,  
\begin{align*}
\Tr(\varphi(T)[T^*,T])&=\Tr(\varphi_1(|T|)(|T|^2-|T^*|^2)+\varphi_2(|T^*|)(|T^*|^2-|T|^2)) \\
&=\Tr(\varphi_1(|T|)(|T|^2-U|T|^2U^*)+\varphi_2(|T|)(|T|^2-U^*|T|^2U)).  
\end{align*}
The previous lemma shows $\Tr(\varphi(T)[T^*,T])\geq 0$, and equality holds if and only if 
$U$ commutes with both $\varphi_1(|T|)$ and $\varphi_2(|T|)$. 
Since $\varphi_1+\varphi_2$ is strictly increasing, it is the case if and only if $U$ commutes 
with $|T|$, which is equivalent to $T$ being normal. 

The second statement follows from 
\[\Tr(T^*[\varphi(T),T])=-\Tr(\varphi(T)[T^*,T]).\] 
\end{proof}

\begin{lemma}\label{decreasing}
Let $\varphi=(\varphi_1,\varphi_2)$ be a pair of functions satisfying the condition (C0), 
and assume $T\in \M_n[a,b]$.  
Let $X(t)$ be a solution of Eq.(\ref{ODE}) defined on $[0,\tau]$. 
Then $\|X(t)\|_{2k}$ for $k\in \N$ and $\|X(t)\|$ are decreasing with respect to $t$ on $[0,\tau]$.  
\end{lemma}

\begin{proof}
We follow the argument in \cite[Section 2]{HS2009}. 
A direct computation with Lemma \ref{rearrangement} shows
\begin{align*}
	\frac{d\|X(t)\|_{2k}^{2k}}{dt} &=\frac{d\Tr((X(t)^*X(t))^{k})}{dt}\\
	&=k\Tr((X(t)^*X(t))^{k-1}(X'(t)^*X(t)+X(t)^*X'(t)))\\
	&=2k\Re \Tr((X(t)^*X(t))^{k-1}X(t)^*[\varphi(X(t)),X(t)])\\
	&=-2k\Tr(\varphi(X(t))(|X(t)|^{2k}-|X(t)^*|^{2k}))\\
	&=-2k\Tr(\varphi_1(|X(t)|)(|X(t)|^{2k}-|X(t)^*|^{2k}))\\
	&-2k\Tr(\varphi_2(|X(t)^*|)(|X(t)^*|^{2k}-|X(t)|^{2k}))\\
	&\leq 0,
\end{align*}
and $\|X(t)\|_{2k}$ is decreasing on $[0,\tau]$ for every $k\in \N$. 
As 
\[\|X(t)\|=\lim_{k\to\infty}\|X(t)\|_{2k},\]
$\|X(t)\|$ is decreasing too. 
\end{proof}

\begin{cor}\label{increasing}
Let $\varphi=(\varphi_1,\varphi_2)$ be a pair of functions satisfying the condition (C0), 
and assume $a>0$ and $T\in \M_n[a,b]$. 
Let $X(t)$ be a solution of Eq.(\ref{ODE}) defined on $[0,\tau]$. 
Then the last singular number $s_n(X(t))$ is increasing on $[0,\tau]$. 
\end{cor}

\begin{proof}
Let $\varphi^{\sim}=(-\varphi_2(x^{-1}),-\varphi_1(x^{-1}))$. 
Then $X(t)^{-1}$ is a solution of Eq.(\ref{ODE}) with $\varphi^\sim$ in place of $\varphi$ 
and $T^{-1}$ in place of $T$. 
Thus $s_n(X(t))=\|X(t)^{-1}\|^{-1}$ is increasing on $[0,\tau]$.  
\end{proof}

\begin{theorem}\label{exi&uni}
Let $\varphi=(\varphi_1,\varphi_2)$ be a pair of functions satisfying the condition (CL), 
and assume $T\in \M_n[a,b]$. 
Then a unique global solution of Eq.(\ref{ODE}) on $[0,\infty)$ exists. 	
\end{theorem}

\begin{proof}
First, assume $a>0$. 
We choose and fix $0<\varepsilon <a$, and extend $\varphi_1$ and $\varphi_2$ to $[a-\varepsilon,b+\varepsilon]$ 
such that they satisfy the condition (CL) for $[a-\varepsilon,b+\varepsilon]$. 
Let $L$ be the Lipschitz constant of $[\varphi(X),X]$ on $\M_n[a-\varepsilon,b+\varepsilon]$, 
and let 
\[M=\sup\{\|[\varphi(X),X]\|;\;X\in \M_n[a,b]\}.\]
Then $T_1\in \M_n$ with $\|T-T_1\|<\varepsilon$ assures $\sigma(|T_1|)\subset (a-\varepsilon,b+\varepsilon)$ 
since 
\[|s_k(T_1)-s_k(T)|\leq \|T_1-T\|,\]
holds for all $1\leq k\leq n$. 
For such $T_1$, we have the estimate 
\[ \|[\varphi(T_1),T_1]\|\leq L\|T_1-T\|+\|[\varphi(T),T]\|\leq L\varepsilon+M.\]
Let $\tau=\varepsilon/(L\varepsilon+M)$. 
Then there exists a unique solution of Eq.(\ref{ODE}) on $[0,\tau]$ 
(see \cite[Theorem 2.2]{Te2012} for example). 
Thanks to Lemma \ref{decreasing} and Corollary \ref{increasing}, we have 
$X(t)\in \M_n[a,b]$ for all $t\in [0,\tau]$. 
Thus the solution uniquely extends to $t\in [0,2\tau]$. 
Repeating the same argument, we get the desired global solution.  

When $a=0$, we can apply the same argument to $[0,b+\varepsilon]$ in place of 
$[a-\varepsilon,b+\varepsilon]$ by using Lemma \ref{decreasing}. 
\end{proof}

\begin{definition}
Let $\varphi=(\varphi_1,\varphi_2)$ be a pair of functions satisfying the condition (CL), 
and assume $T\in \M_n[a,b]$. 
Let $X(t)$ be the unique global solution of Eq.(\ref{ODE}) on $[0,\infty)$.  
We define a flow $F^\varphi_t$ in $\M_n$ by $F^\varphi_t(T)=X(t)$. 

We denote $F^A_t=F^{(\log x,0)}_t$ and $F^H_t=F^{(x^2,x^2)}_t$, and call them 
\textit{the Aluthge flow} and \textit{the Haagerup flow} respectively. 
\end{definition}

\begin{proposition}\label{orbit}
Let $\varphi=(\varphi_1,\varphi_2)$ be a pair of functions satisfying the condition (CL), 
and assume $T\in \M_n[a,b]$. 
Then the flow $\{F^\varphi_t(T)\}_{t\geq 0}$ stays in the similarity orbit $\cO(T)$ of $T$. 
\end{proposition}

\begin{proof} Note that we have $[\varphi(X),X]\in T_X\cO(T)$ for $X\in \M_n[a,b]\cap \cO(T)$ as 
\[[\varphi(X),X]=\frac{d}{ds}e^{s\varphi(X)}Xe^{-s\varphi(X)}|_{s=0},\] 
where $T_X\cO(T)$ denotes the tangent space of $\cO(T)$ at $X$. 
Thus the uniqueness of the solution yields that the flow is the integral curve of 
the vector field $[\varphi(X),X]$ in $\cO(T)$. 

Or alternatively, we can follow the argument in \cite[Section 2]{HS2009}. 
We solve the two linear equations in $\M_n$: 
\begin{equation*}
\frac{dV(t)}{dt}=-V(t)\varphi(F^\varphi_t(T)),\quad V(0)=I,
\end{equation*} 
\begin{equation*}
\frac{dW(t)}{dt}=\varphi(F^\varphi_t(T))W(t),\quad W(0)=I, 
\end{equation*}
which have unique global solutions. 
A direct computation shows 
\[\frac{d(V(t)W(t))}{dt}=0,\]
\[\frac{d(V(t)F^\varphi_t(T)W(t))}{dt}=0,\]
which imply $V(t)W(t)=I$ and $V(t)F^\varphi_t(T)W(t)=T$. 
Thus $F^\varphi_t(T)=V(t)^{-1}TV(t)$. 
\end{proof}

\begin{remark} The Aluthge flow has a special property with respect to tensor product. 
For $T_i\in GL(n_i,\C)$, $i=1,2,\cdots,m$, we have 
\[\log\left| T_1\otimes T_2\otimes\cdots\otimes T_m \right|
=\sum_{i=1}^m I_{n_1}\otimes\cdots\otimes I_{n_{i-1}}\otimes \log|T_i|\otimes I_{n_{i+1}}\otimes\cdots\otimes I_{n_m}. \]
Consequently, we have 
\[F^A_t(T_1\otimes T_2\otimes \cdots \otimes T_m)
=F^A_t(T_1)\otimes F^A_t(T_2)\otimes \cdots \otimes F^A_t(T_m).\]
For $T\in GL(n,\C)$, we denote by $T^{\wedge^k}$ the restriction of $T^{\otimes^k}$ 
to the $k$-th antisymmetric tensor product $\bigwedge^k \C^n$ of $\C^n$. 
Then we get $F_t^A(T)^{\wedge^k}=F_t^A(T^{\wedge^k})$. 
Thus Lemma \ref{decreasing} shows that $\|F_t^A(T)^{\wedge^k}\|$ is decreasing with respect to $t$, 
which means that 
\[\prod_{j=1}^ks_j(F^A_t(T))\] 
is decreasing for every $1\leq k\leq n$. 
When $k=n$, the above product equals $|\det F^A_t(T)|=|\det T|$ since $F^A_t(T)\in \cO(T)$, and it is a constant. 
Thus $F^A_t(T)$ is decreasing in the sense of log-majorization (see \cite{AH1994}). 
In particular, $|||F^A_t(T)|||$ is decreasing for every unitarily invariant norm $|||\cdot|||$. 
\end{remark}

\begin{question} Is $|||F^\varphi_t(T)|||$ decreasing with respect to $t$ 
for any unitarily invariant norm $|||\cdot|||$? 
We have already seen in Lemma \ref{decreasing} that it is the case for the Schatten norm $\|\cdot \|_{2k}$ 
for $k\in \N$.  
\end{question}

\section{Convergence results}
Since the flow $F^\varphi_t$ satisfies $F^\varphi_t(UTU^{-1})=UF^\varphi_t(T)U^{-1}$ for any unitary $U\in \M_n$,   
by the Schur unitary triangulation theorem, we may assume 
that the initial value $T$ is upper triangular in order to obtain convergence results.  
Throughout this section, we fix $\varphi$ satisfying the condition (CL) 
and fix an upper triangular matrix $T\in \M_n$ with $T\in \M_n[a,b]$. 
We denote by $\Lambda=\diag(\lambda_1,\lambda_2,\cdots,\lambda_n)$ the diagonal part of $T$. 

Our basic idea is to separate the dynamics of $F^\varphi_t(T)$ into that of 
upper triangular matrices and that of unitary matrices.  
For this purpose, we need to find a unitary flow $U(t)$ satisfying $U(t)F^\varphi_t(T)U(t)^{-1}=\Lambda+Y(t)$ 
with strictly upper triangular $Y(t)$. 
Assuming such $U(t)$ exists, we would get 
\begin{align*}
	\frac{d (U(t)F^\varphi_t(T)U(t)^{-1})}{dt} &=\frac{dU(t)}{dt}F^\varphi_t(T)U(t)^{-1}
	+U(t)[\varphi(F^\varphi_t(T)),F^\varphi_t(T)]U(t)^{-1}\\
	&-U(t)F^\varphi_t(T)U(t)^{-1}\frac{dU(t)}{dt}U(t)^{-1}\\
	&=[\frac{dU(t)}{dt}U(t)^{-1}+\varphi(\Lambda+Y(t)),\Lambda+Y(t)].
\end{align*}
Since $U(t)$ is required to be unitary, the term $\frac{dU(t)}{dt}U(t)^{-1}$ should be skew-adjoint. 
On the other hand, to keep the commutator at the end strictly upper triangular, the term 
$\frac{dU(t)}{dt}U(t)^{-1}+\varphi(\Lambda+Y(t))$ should be upper triangular. 
We can fulfill this requirement by using the additive decomposition of $\varphi(\Lambda+Y)$ into 
its strictly upper triangular, diagonal, and strictly lower triangular parts. 

In the following arguments, we treat $Y=(y_{ij})$ as a variable whose value is in the set of $n$-by-$n$ 
strictly upper triangular matrices $\mathbf{N}_n$, while we treat $\Lambda$ as a fixed diagonal matrix. 
We decompose $\varphi(\Lambda +Y)$ as 
\[\varphi(\Lambda +Y)=P(Y)+D(Y)+P(Y)^*,\] 
with strictly upper triangular $P(Y)$ and diagonal $D(Y)$.  
Now we consider the following system of ODEs: 

\begin{equation}\label{upptriflow}
\frac{dY(t)}{dt}=[D(Y(t))+2P(Y(t)),\Lambda+Y(t)],\quad Y(0)=T-\Lambda,
\end{equation}
\begin{equation}\label{unitaryflow}
\frac{dU(t)}{dt}=(P(Y(t))-P(Y(t))^*)U(t),\quad U(0)=I.
\end{equation}

\begin{lemma}
Let the notation be as above. 
Then unique global solutions of Eq.(\ref{upptriflow}) and Eq.(\ref{unitaryflow}) exist on $[0,\infty)$, 
and $U(t)$ is unitary for all $t\in [0,\infty)$. 
Moreover, we have 
\[F^\varphi_t(T)=U(t)^{-1}(\Lambda+Y(t))U(t),\]
for all $t\in [0,\infty)$. 
\end{lemma}

\begin{proof} We show the statement assuming $a>0$. 
The case $a=0$ can be treated in a similar way. 
We choose $0<\varepsilon<a$ and extend $\varphi_1$ and $\varphi_2$ to $[a-\varepsilon,b+\varepsilon]$ 
satisfying the condition (CL). 
Let $L$ be the Lipschitz constant of $[D(Y)+2P(Y),\Lambda+Y]$ on 
\[\{Y\in \mathbf{N}_n;\; \Lambda +Y\in \M_n[a-\varepsilon,b+\varepsilon] \},\]  
and let 
\[M=\sup_{\Lambda+Y\in \M_n[a,b]}\|[D(Y)+2P(Y),\Lambda+Y]\|.\]
Then a unique local solution of Eq.(\ref{upptriflow}) exists on $[0,\tau]$ with 
$\tau=\varepsilon/(L\varepsilon+M)$ as in the proof of Theorem \ref{exi&uni}.   
Since Eq.(\ref{unitaryflow}) is a linear equation with a continuous coefficient term $(P(Y(t))-P(Y(t))^*)$, 
a unique solution exists on $[0,\tau]$ too. 
A direct computation shows 
\[\frac{d(U(t)^*U(t))}{dt}=0,\]
and $U(t)$ is unitary. 

Let $X(t)=U(t)^{-1}(\Lambda+Y(t))U(t)$ for $t\in [0,\tau]$. 
Then $X(0)=T$ and 
\begin{align*}
\frac{dX(t)}{dt}&=-U(t)^{-1}\frac{dU(t)}{dt}U(t)^{-1}(\Lambda+Y(t))U(t)+U(t)^{-1}\frac{dY(t)}{dt}U(t) \\
&+U(t)^{-1}(\Lambda+Y(t))\frac{dU(t)}{dt}\\
&=-U(t)^{-1}(P(Y(t))-P(Y(t))^*)(\Lambda+Y(t))U(t)\\
&+U(t)^{-1}[D(Y(t))+2P(Y(t)),\Lambda+Y(t)]U(t)\\
&+U(t)^{-1}(\Lambda+Y(t))(P(Y(t))-P(Y(t))^*)U(t)\\
&=U(t)^{-1}[\varphi(\Lambda+Y(t)),\Lambda+Y(t)]U(t)\\
&=[\varphi(X(t)),X(t)]. 
\end{align*} 
Thus Theorem \ref{exi&uni} implies $U(t)^{-1}(\Lambda+Y(t))U(t)=F^\varphi_t(T)$. 
This shows $\Lambda +Y(t)\in \M_n[a,b]$ for all $t\in [0,\tau]$, and we can extend the solution 
$Y(t)$ to $[0,2\tau]$. 
Repeating the same argument, we get the desired unique global solutions of 
Eq.(\ref{upptriflow}) and Eq.(\ref{unitaryflow}), which satisfy
$F^\varphi_t(T)=U(t)^{-1}(\Lambda+Y(t))U(t)$.   
\end{proof}

Since 
\[\|F^\varphi_t(T)-U(t)^{-1}\Lambda U(t)\|_2=\|Y(t)\|_2,\]
our task is to show that $\{Y(t)\}_{t\geq 0}$ converges to 0 and 
$\{U(t)^{-1}\Lambda U(t)\}_{t\geq 0}$ converges as $t$ tends to $\infty$. 
The first convergence holds without any additional assumption.

\begin{theorem}\label{Lconvergence} Let the notation be as above. 
Then $\|Y(t)\|_2$ converges to 0 as $t$ tends to $\infty$.  
\end{theorem}

\begin{proof} We show that $\|Y\|_2^2$ is a strict Lyapunov function (see \cite[page 201]{Te2012} for the definition). 
Since $\|F^\varphi_t(T)\|_2^2=\|\Lambda\|_2^2+\|Y(t)\|_2^2$, Lemma \ref{positive semi-definite} and 
Lemma \ref{decreasing} show 
\begin{align}\label{Lyapunov}
\frac{d\|Y(t)\|_2^2}{dt}&=\frac{d\|F^\varphi_t(T)\|_2^2}{dt} \\ \nonumber
 &= -2\Tr(\varphi(F^\varphi_t(T))[F^\varphi_t(T)^*,F^\varphi_t(T)])\\ \nonumber
 &=-2\Tr(\varphi(\Lambda+Y(t))[(\Lambda+Y(t))^*,\Lambda+Y(t)])\\ \nonumber
 &\leq 0,
\end{align}
and equality holds if and only if $\Lambda+Y(t)$ is normal, that is, $Y(t)=0$. 
Thus the convergence follows from \cite[Theorem 6.13, 6.14]{Te2012}.  
\end{proof}

As a corollary, we obtain an analogue of Yamazaki's result \cite[Theorem 1]{Y2002} for iterated  Aluthge transforms (see also \cite{Ta2008}). 

\begin{cor}\label{numerical} Under the condition (CL), the operator norm $\|F^\varphi_t(T)\|$ converges to the spectral 
radius $r(T)$ of $T$ as $t$ tends to $\infty$. 

\end{cor}

\begin{cor}\label{sca+nil} Under the condition (CL), if $T$ is a scalar $\alpha I$ plus a nilpotent matrix, then  
$F^\varphi_t(T)$ converges to $\alpha I$ as $t$ tends to $\infty$.  
\end{cor}

Note that we have 
\begin{align*}
\left\|\frac{d(U(t)^{-1}\Lambda U(t))}{dt}\right\|_2
 &=\left\| -U(t)^{-1}\frac{dU(t)}{dt}U(t)^{-1}\Lambda U(t)+U(t)^{-1}\Lambda \frac{dU(t)}{dt}\right\|_2 \\
 &=\|[\Lambda,P(Y(t))-P(Y(t))^*]\|_2, 
\end{align*}
and 
\[2\|[\Lambda,P(Y(t))]\|_2^2= \|[\Lambda,\varphi(\Lambda+Y(t))]\|_2^2
=\|[\Lambda, \varphi(\Lambda+Y(t))-\varphi(\Lambda)]\|_2^2.\]
Thus to prove the convergence of $\{U(t)^{-1}\Lambda U(t)\}_{t\geq 0}$, it suffices to show 
\begin{equation}\label{finitelength}
\sqrt{2}\int_0^\infty  \|[\Lambda,P(Y(t))]\|_2dt
=\int_0^\infty \|[\Lambda, \varphi(\Lambda+Y(t))-\varphi(\Lambda)]\|_2dt<\infty. 
\end{equation}
Since $\varphi$ is Lipschitz, this follows if $Y(t)$ decays sufficiently fast.  

To perform further analysis, we may and do arrange $T$, up to unitary equivalence, so that 
\[\Lambda=\diag(\overbrace{\mu_1,\cdots,\mu_1}^{m^a_1},\overbrace{\mu_2,\cdots,\mu_2}^{m^a_2},\cdots, 
\overbrace{\mu_s,\cdots,\mu_s}^{m^a_s}),\]
\[|\mu_1|\leq |\mu_2|\leq \cdots\leq |\mu_s|,\]
with distinct eigenvalues $\{\mu_1,\mu_2,\cdots,\mu_s\}$ of $T$. 
Thus if $T$ is not invertible, we have $\mu_1=0$. 
We call $m^a_k$ the algebraic multiplicity of the eigenvalue $\mu_k$. 
We set $m^g_k=\dim \ker(T-\mu_k I)$, and call it the geometric multiplicity of $\mu_k$. 
When $\mu_1=0$, we may and do assume $T_{ij}=0$ for $1\leq i\leq j\leq m^g_1$. 

\begin{remark}\label{diagonalizable} 
Under the above assumption, the following hold: 
\begin{itemize}
\item[(1)] If $T$ as above is diagonalizable, then $T_{ij}=0$ for all $i<j$ with $\lambda_i=\lambda_j$.   
	Indeed, if there exist two indices $i<j$ with $\lambda_i=\lambda_j$ and $T_{ij}\neq 0$, 
	we can express $T-\lambda_iI$ as the block matrix 
	\[T-\lambda_iI=\left(
	\begin{array}{ccc}
		A &B &*  \\
		O &N &*  \\
		O &O &C 
	\end{array}
	\right),
	\]
	so that $A$ and $C$ are invertible matrices and $N$ is a non-zero nilpotent matrix. 
	Since \[(T-\lambda_iI)^2=\left(
	\begin{array}{ccc}
		A^2 &AB+ BN&*  \\
		O &N^2 &*  \\
		O &O &C^2 
	\end{array}
	\right),
	\]
	we get $\ker (T-\lambda_i I)^2\varsupsetneqq\ker (T-\lambda_i I)$, and thus $T$ is not diagonalizable. 
\item[(2)] If $\mu_1=0$, and $Y(t)=(y(t)_{ij})$ is the unique solution of Eq.(\ref{upptriflow}), 
we have $y_{ij}(t)=0$ for all $1\leq i<j\leq m^g_1$ and $t\geq 0$. 
Indeed, it follows from the uniqueness of the solution since we can show that Eq.(\ref{upptriflow}) with 
this additional condition has a solution. 
In the rest of this section, we always assume $y_{ij}=0$ for $1\leq i<j\leq m_1^g$. 
\end{itemize}
\end{remark}	

For $f\in C^1[a,b]$, we define its first divided difference $f^{[1]}\in C([a,b]^2)$ by 
\[f^{[1]}(x,y)=\left\{
\begin{array}{ll}
\frac{f(x)-f(y)}{x-y} , &\quad x\neq y,  \\
f'(x) , &\quad x=y. 
\end{array}
\right.
\]
We denote by $\{e_i\}_{i=1}^n$ the canonical basis of $\C^n$, and by $\{E_{ij}\}_{1\leq i,j\leq n}$ 
the system of matrix units corresponding to $\{e_i\}_{i=1}^n$.  

\begin{lemma}\label{Hessian} 
Let the notation be as above, and assume that $\varphi$ satisfies the condition (C1). 
Then 
\begin{align*}\lefteqn{\Tr(\varphi(\Lambda+Y)[(\Lambda+Y)^*,\Lambda+Y])}\\
 &= \sum_{i<j, \lambda_i\neq \lambda_j}
\frac{\varphi_1^{[1]}(|\lambda_i|,|\lambda_j|)+\varphi_2^{[1]}(|\lambda_i|,|\lambda_j|)}{|\lambda_i|+|\lambda_j|}
|\lambda_i-\lambda_j|^2|y_{ij}|^2
+o(\|Y\|_2^2)
\end{align*}
as $Y$ tends to 0. 
\end{lemma}

\begin{proof}
A direct computation shows 
\begin{align*}
\lefteqn{\Tr(\varphi(\Lambda+Y)[(\Lambda+Y)^*,\Lambda+Y])}\\
& =\Tr(\varphi(\Lambda)([\Lambda^*,Y]+[Y^*,\Lambda]+[Y^*,Y])) \\
&+\Tr((\varphi(\Lambda+Y)-\varphi(\Lambda))([\Lambda^*,Y]+[Y^*,\Lambda]+[Y^*,Y])). 
\end{align*}
The first term is equal to 
\[\Tr([\varphi(\Lambda),Y^*]Y)=\sum_{i<j}\left(\varphi_1(|\lambda_j|)-\varphi_1(|\lambda_i|)
-\varphi_2(|\lambda_j|)+\varphi_2(|\lambda_i|)\right)|y_{ij}|^2.\]
The second term is 
\[2\Re\sum_{i<j}(\varphi(\Lambda+Y)-\varphi(\Lambda))_{ij}(\lambda_i-\lambda_j)\overline{y_{ij}}
+o(\|Y\|_2^2).\]
Thus \begin{align*}
	\lefteqn{\Tr(\varphi(\Lambda+Y)[(\Lambda+Y)^*,\Lambda+Y])}\\
	& =\sum_{i<j}[(\varphi_1(|\lambda_j|)-\varphi_1(|\lambda_i|)) |y_{ij}|^2
	+2\Re\left( (\varphi_1(|\Lambda+Y|)-\varphi_1(|\Lambda|))_{ij}(\lambda_i-\lambda_j)\overline{y_{ij}}\right)]\\
	&- \sum_{i<j}[(\varphi_2(|\lambda_j|)-\varphi_2(|\lambda_i|)) |y_{ij}|^2
	+2\Re\left( (\varphi_2(|\Lambda^*+Y^*|)-\varphi_2(|\Lambda|))_{ij}(\lambda_i-\lambda_j)\overline{y_{ij}}\right)].
\end{align*}

If $a\neq 0$, the rest of the computation follows from the usual Daleckii-Krein formula 
(see, for example, \cite[Theorem 3.33]{HP2014}, \cite[Theorem 2.3.1]{H2010}) 
applied to $\varphi_1(\sqrt{x})$ and $\varphi_2(\sqrt{x})$.  
However, if $a=0$, since $\varphi_1(\sqrt{x})$ and $\varphi_2(\sqrt{x})$ are not necessarily differentiable at $0$, 
we need a more careful argument.     
Assume $1\leq i<j\leq n$. 
We further assume $\lambda_i\neq \lambda_j$ since only such pairs $(i,j)$ have contributions to the above summation. 
Note that $\lambda_j\neq 0$ from our convention. 
Let 
\[|\Lambda+Y|=\sum_{k=1}^m \nu_k F_k\]
be the spectral decomposition. 
Then we have 
\[\varphi_1(|\Lambda+Y|)-\varphi_1(|\Lambda|)=\sum_{\nu_k\neq |\lambda_l|}
\frac{\varphi_1(\nu_k)-\varphi_1(|\lambda_l|)}{\nu_k^2-|\lambda_l|^2}F_k(|\Lambda+Y|^2-|\Lambda|^2)E_{ll}.\]
Since $F_k(|\Lambda+Y|^2-|\Lambda|^2)E_{jj}=0$ for $\nu_k=|\lambda_j|$, we get 
\[(\varphi_1(|\Lambda+Y|)-\varphi_1(|\Lambda|))_{ij}=\sum_{k=1}^m\frac{\varphi^{[1]}_1(\nu_k,|\lambda_j|)}{\nu_k+|\lambda_j|}
\inpr{F_k(|\Lambda+Y|^2-|\Lambda|^2)e_j}{e_i}.\]
Letting
\[f_j(x)=\frac{\varphi^{[1]}_1(x,|\lambda_j|)}{x+|\lambda_j|},\]
which is continuous on $[a,b]$, we get 
\begin{align*}
(\varphi_1(|\Lambda+Y|)-\varphi_1(|\Lambda|))_{ij}
	&=\inpr{f_j(|\Lambda+Y|)(|\Lambda+Y|^2-|\Lambda|^2)e_j}{e_i}  \\
	&=\inpr{f_j(|\Lambda+Y|)(\Lambda^*Y+Y^*\Lambda)e_j}{e_i}+o(\|Y\|_2)  \\
	&=\inpr{f_j(|\Lambda|)(\Lambda^*Y+Y^*\Lambda)e_j}{e_i}+o(\|Y\|_2)  \\
	&=\frac{\varphi^{[1]}_1(|\lambda_i|,|\lambda_j|)}{|\lambda_i|+|\lambda_j|}\overline{\lambda_i}y_{ij}+o(\|Y\|_2). 
\end{align*}
This shows 
\begin{align*}
	&(\varphi_1(|\lambda_j|)-\varphi_1(|\lambda_i|)) |y_{ij}|^2
	+2\Re\left( (\varphi_1(|\Lambda+Y|)-\varphi_1(|\Lambda|))_{ij}(\lambda_i-\lambda_j)\overline{y_{ij}}\right)\\
	&=\left(\varphi_1(|\lambda_j|)-\varphi_1(|\lambda_i|) 
	+\frac{\varphi^{[1]}_1(|\lambda_i|,|\lambda_j|)}{|\lambda_i|+|\lambda_j|}
	(2|\lambda_i|^2-\lambda_i\overline{\lambda_j}-\overline{\lambda_i}\lambda_j)\right)|y_{ij}|^2
	+o(\|Y\|_2^2) \\
	&=\frac{\varphi^{[1]}_1(|\lambda_i|,|\lambda_j|)}{|\lambda_i|+|\lambda_j|}
	|\lambda_i-\lambda_j|^2|y_{ij}|^2+o(\|Y\|_2^2).
	\end{align*}
	
In a similar way, we can show 
\[	(\varphi_2(|\Lambda^*+Y^*|)-\varphi_2(|\Lambda^*|))_{ij}
=\frac{\varphi^{[1]}_2(|\lambda_i|,|\lambda_j|)}{|\lambda_i|+|\lambda_j|}\overline{\lambda_j}y_{ij}+o(\|Y\|_2), \]
and 
\begin{align*}
&(\varphi_2(|\lambda_j|)-\varphi_2(|\lambda_i|)) |y_{ij}|^2
+2\Re\left( (\varphi_2(|\Lambda^*+Y^*|)-\varphi_2(|\Lambda|))_{ij}(\lambda_i-\lambda_j)\overline{y_{ij}}\right)\\
&=\left(\varphi_2(|\lambda_j|)-\varphi_2(|\lambda_i|)
+\frac{\varphi^{[1]}_2(|\lambda_i|,|\lambda_j|)}{|\lambda_i|+|\lambda_j|}
(\lambda_i\overline{\lambda_j}+\overline{\lambda_i}\lambda_j-2|\lambda_j|^2)\right) |y_{ij}|^2+o(\|Y\|_2^2)\\
&=-\frac{\varphi^{[1]}_2(|\lambda_i|,|\lambda_j|)}{|\lambda_i|+|\lambda_j|}
|\lambda_i-\lambda_j|^2|y_{ij}|^2+o(\|Y\|_2^2).
\end{align*}
Therefore the statement holds. 
\end{proof}

The following is a counterpart of the main result in \cite{APS2007}.

\begin{theorem}\label{exponential} Let the notation be as above, and assume that $\varphi$ satisfies the condition (C1) and 
	$T$ is diagonalizable. 
Then the flow $\{F^\varphi_t(T)\}$ converges to a normal matrix in $\cO(T)$ exponentially fast as $t$ tends to $\infty$.  
\end{theorem}

\begin{proof}
We first show that $Y(t)$ converges exponentially fast. 
Thanks to Remark \ref{diagonalizable}, we have $y_{ij}(t)=0$ if $\lambda_i=\lambda_j$. 
Theorem \ref{Lconvergence}, Eq.(\ref{Lyapunov}), and Lemma \ref{Hessian} show that there exist $C>0$ and $\tau>0$ such that 
\[\frac{d\|Y(t)\|_2^2}{dt}\leq -C\|Y(t)\|_2^2,\quad \forall t\geq \tau,\]
which implies $\|Y(t)\|_2^2\leq e^{-C(t-\tau)}\|Y(\tau)\|_2^2$ for all $t\geq \tau$. 

The Lipschitz continuity of $P$ implies that $\|\frac{dU(t)}{dt}\|_2$ converges to 0 exponentially fast, 
and $U(t)$ converges to a unitary matrix, say $U(\infty)$ exponentially fast. 
Therefore $F^\varphi_t(T)$ converges to $U(\infty)^{-1}\Lambda U(\infty)$ exponentially fast.  
\end{proof}

\begin{example} When $n=2$, Corollary \ref{sca+nil} and Theorem \ref{exponential} show that $\{F^\varphi_t(T)\}_{t\geq 0}$ 
always converges to a normal matrix under the condition (C1). 
In this case, $Y$ has only one variable $y_{12}$, which we denote by $y$ for simplicity. 
We can explicitly write down Eq.(\ref{upptriflow}) in terms of $y$ as follows. 
We denote $s_i(\Lambda+Y)=s_i(y)$ for simplicity. 
Then they are given by  
\[s_1(y)=\sqrt{\frac{|\lambda_1|^2+|\lambda_2|^2+|y|^2+\sqrt{\left((|\lambda_1|+|\lambda_2|)^2+|y|^2\right)
\left((|\lambda_1|-|\lambda_2|)^2+|y|^2\right)}}{2}},\]
\[s_2(y)=\sqrt{\frac{|\lambda_1|^2+|\lambda_2|^2+|y|^2-\sqrt{\left((|\lambda_1|+|\lambda_2|)^2+|y|^2\right)
\left((|\lambda_1|-|\lambda_2|)^2+|y|^2\right)}}{2}}.\]
We can compute $\varphi(\Lambda+Y)$ by using the interpolation polynomial (see Section \ref{interpolation}), and obtain 
\begin{align*}
\frac{dy}{dt}	 & =-\frac{\varphi_1^{[1]}(s_1(y),s_2(y))}{s_1(y)+s_2(y)}(|\lambda_1|^2+|\lambda_2|^2-2\overline{\lambda_1}\lambda_2+|y|^2)y\\
	& -\frac{\varphi_2^{[1]}(s_1(y),s_2(y))}{s_1(y)+s_2(y)}(|\lambda_1|^2+|\lambda_2|^2-2\lambda_1\overline{\lambda_2}+|y|^2)y.
\end{align*}
Letting $y(t)=|y(t)|e^{i\theta(t)}$ with $\theta(t)\in \R$, we get 
\[\frac{d|y|}{dt}=-\frac{\varphi_1^{[1]}(s_1(y),s_2(y))+\varphi_2^{[1]}(s_1(y),s_2(y))}{s_1(y)+s_2(y)}(|\lambda_1-\lambda_2|^2+|y|^2)|y|, \]
\[\frac{d\theta}{dt}=2\frac{\varphi_1^{[1]}(s_1(y),s_2(y))-\varphi_2^{[1]}(s_1(y),s_2(y))}{s_1(y)+s_2(y)}\Im \overline{\lambda_1}\lambda_2.
\]
In the case of the Haagerup flow, we get the following very simple ODE: 
\[\frac{dy}{dt}=-2(|\lambda_1-\lambda_2|^2+|y|^2)y.\]
In particular, if $\lambda_1=\lambda_2$, 
\[y(t)=\frac{y(0)}{\sqrt{1+4|y(0)|^2t}},\]
which shows how slow the convergence is in the case of non-diagonalizable $T$. 
\end{example}

In the rest of this section, we show our general convergence result under extra assumptions.
Our goal is to show the convergence (\ref{finitelength}). 
For this purpose, we need to separate fast converging variables $y_{ij}(t)$ with $\lambda_i\neq \lambda_j$ 
from the others. 
Our Lyapunov-like function at this time is 
\[\|[\Lambda, Y]\|_2^2=\sum_{i<j}|\lambda_i-\lambda_j|^2|y_{ij}|^2.\] 
Note that we have 
\begin{align}\label{Lfderivative}
\frac{d\|[\Lambda,Y(t)]\|_2^2}{dt}	 &=2\Re\Tr([\Lambda,Y(t)]^*[\Lambda,\frac{dY(t)}{dt}])  \\
&= -2\Re\Tr([\Lambda^*,Y(t)^*][\Lambda,[D(Y(t))+2P(Y(t)), \Lambda+Y(t)]]). \nonumber
\end{align}

In the proof of the previous theorem, we have seen that the following holds under the condition (C1) for 
$i<j$ with $\lambda_i\neq \lambda_j$:  
\begin{align}\label{derivative} 
\lefteqn{[\Lambda,\varphi(\Lambda+Y)]_{ij}} \\
	 &=(\lambda_i-\lambda_j)\left(  \frac{\varphi^{[1]}_1(|\lambda_i|,|\lambda_j|)}{|\lambda_i|+|\lambda_j|}\overline{\lambda_i}
	 -\frac{\varphi^{[1]}_2(|\lambda_i|,|\lambda_j|)}{|\lambda_i|+|\lambda_j|}\overline{\lambda_j} \right)y_{ij}+o(\|Y\|_2).\nonumber
\end{align}
To obtain a general convergence result, all we have to do is to turn the error term into $o(\|[\Lambda,Y]\|_2)$ 
as $Y\to 0$.

\begin{lemma}\label{convergence}
Let the notation be as above and assume that $\varphi$ satisfies the condition (C1). 
If 
\begin{align}\label{criterion}
\lefteqn{[\Lambda,\varphi(\Lambda+Y)]_{ij}} \\
	 &=(\lambda_i-\lambda_j)\left(  \frac{\varphi^{[1]}_1(|\lambda_i|,|\lambda_j|)}{|\lambda_i|+|\lambda_j|}\overline{\lambda_i}
	 -\frac{\varphi^{[1]}_2(|\lambda_i|,|\lambda_j|)}{|\lambda_i|+|\lambda_j|}\overline{\lambda_j}\right)y_{ij}
	 +o(\|[\Lambda,Y]\|_2),\quad (Y\to 0),\nonumber
\end{align}
holds for all $i<j$ with $\lambda_i\neq \lambda_j$, the flow $\{F^\varphi_t(T)\}_{t\geq 0}$ converges to a normal matrix 
as $t$ tends to $\infty$ for any $T\in \M_n[a,b]$.  
\end{lemma}

\begin{proof} First we claim 
\begin{align*}
	&\Re \Tr([\Lambda^*,Y^*][\Lambda,[D(Y)+2P(Y), \Lambda+Y]])  \\
	&=\sum_{i<j,\;\lambda_i\neq \lambda_j}
	\frac{\varphi^{[1]}_1(|\lambda_i|,|\lambda_j|)+\varphi^{[1]}_2(|\lambda_i|,|\lambda_j|)}{|\lambda_i|+|\lambda_j|}
	|\lambda_i-\lambda_j|^4|y_{ij}|^2+o(\|[\Lambda,Y]\|_2^2).\\
\end{align*}
A direct computation with our assumption shows 
\begin{align*}
	 &\Tr([\Lambda^*,Y^*][\Lambda,[D(Y)+2P(Y), \Lambda+Y]])  \\
	& =\Tr([\Lambda^*,Y^*](2[[\Lambda,P(Y)], \Lambda+Y]+
[D(Y)+2P(Y), [\Lambda,Y]]	))\\
&=2\Tr([\Lambda,[\Lambda^*,Y^*]][\Lambda,P(Y)])+2\Tr([\Lambda^*,Y^*][[\Lambda,P(Y)],Y])\\
&+\Tr([\Lambda^*,Y^*][D(Y),[\Lambda,Y]])+
2\Tr([\Lambda^*,Y^*][P(Y),[\Lambda,Y]])\\
&=2\Tr([\Lambda,[\Lambda^*,Y^*]][\Lambda,P(Y)])+\Tr([\Lambda^*,Y^*][\varphi(\Lambda),[\Lambda,Y]])+o(\|[\Lambda,Y]\|_2^2). 
\end{align*}
Thus 
\begin{align*}
	 & \Tr([\Lambda^*,Y^*][\Lambda,[D(Y)+2P(Y), \Lambda+Y]])  \\
	& =2\sum_{i<j,\;\lambda_i\neq \lambda_j}|\lambda_i-\lambda_j|^2[\Lambda, P(Y)]_{ij}\overline{y_{ij}}\\
	&-\sum_{i<j,\;\lambda_i\neq \lambda_j}|\lambda_i-\lambda_j|^2(\varphi(\Lambda)_{ii}-\varphi(\Lambda)_{jj})|y_{ij}|^2
	+o(\|[\Lambda,Y]\|_2^2),
\end{align*}
and 
\begin{align*}
	&\Re \Tr([\Lambda^*,Y^*][\Lambda,[D(Y)+2P(Y), \Lambda+Y]])  \\
	& =\sum_{i<j,\;\lambda_i\neq \lambda_j}|\lambda_i-\lambda_j|^2
	[  \frac{\varphi^{[1]}_1(|\lambda_i|,|\lambda_j|)}{|\lambda_i|+|\lambda_j|}
	(2|\lambda_i|^2-\overline{\lambda_i}\lambda_j-\lambda_i\overline{\lambda_j})\\
	&-\frac{\varphi^{[1]}_2(|\lambda_i|,|\lambda_j|)}{|\lambda_i|+|\lambda_j|}
	(\lambda_i\overline{\lambda_j}+\overline{\lambda_i}\lambda_j-2|\lambda_j|^2)	 
	]|y_{ij}|^2\\
	&-\sum_{i<j,\;\lambda_i\neq \lambda_j}|\lambda_i-\lambda_j|^2(\varphi(\Lambda)_{ii}-\varphi(\Lambda)_{jj})|y_{ij}|^2+o(\|[\Lambda,Y]\|_2^2)\\
	&=\sum_{i<j,\;\lambda_i\neq \lambda_j}
	\frac{\varphi^{[1]}_1(|\lambda_i|,|\lambda_j|)+\varphi^{[1]}_2(|\lambda_i|,|\lambda_j|)}{|\lambda_i|+|\lambda_j|}
	|\lambda_i-\lambda_j|^4|y_{ij}|^2+o(\|[\Lambda,Y]\|_2^2).\\
\end{align*}

The claim with Eq.(\ref{Lfderivative}) shows 
\begin{align*}
 &\frac{d\|[\Lambda,Y]\|_2^2}{dt} \\
 &=-2\sum_{i<j,\;\lambda_i\neq \lambda_j}
	\frac{\varphi^{[1]}_1(|\lambda_i|,|\lambda_j|)+\varphi^{[1]}_2(|\lambda_i|,|\lambda_j|)}{|\lambda_i|+|\lambda_j|}
	|\lambda_i-\lambda_j|^4|y_{ij}|^2+o(\|[\Lambda,Y]\|_2^2). 
\end{align*}
Therefore thanks to Theorem \ref{Lconvergence}, there exist $C>0$ and $\tau>0$ such that for all $t\geq \tau$, we have 
\[\frac{d\|[\Lambda,Y(t)]\|_2^2}{dt}\leq -C\|[\Lambda,Y(t)]\|_2^2,\] 
which implies that  $\|[\Lambda,Y(t)]\|_2$ converges to 0 exponentially fast. 
Since Eq.(\ref{criterion}) shows that $\|[\Lambda,P(Y(t))]\|_2$ converges to 0 exponentially fast too, 
we get the convergence (\ref{finitelength}). 
\end{proof}

\begin{theorem}\label{Cn-1}
Let the notation be as above and assume that 
$\varphi$ satisfies the condition (C1). 
Letting $\psi_i(x)=\varphi_i(\sqrt{x})$ for $i=1,2$, we further assume $\psi_i\in C^{n-1}[a^2,b^2]$ for $i=1,2$ 
(this is equivalent to $\varphi_i\in C^{n-1}[a,b]$ if $a> 0$). 
Then the flow $\{F^\varphi_t(T)\}_{t\geq 0}$ converges to a normal matrix 
as $t$ tends to $\infty$ for any $T\in \M_n[a,b]$.  
\end{theorem}

\begin{proof} It suffices to verify Eq.(\ref{criterion}). 
Note that there exists a polynomial of degree at most $n-1$  
\[p_Y(x)=\sum_{k=0}^{n-1}a_k(Y)x^k\]
with coefficients $a_k(Y)$ continuous functions of $Y$ satisfying 
\[\varphi_1(|\Lambda+Y|)=\psi_1(|\Lambda+Y|^2)=p_Y(|\Lambda+Y|^2),\]  
and necessarily $\psi_1(s_i(\Lambda+Y)^2)=p_Y(s_i(\Lambda+Y)^2)$. 
Moreover, we may further assume ${\psi_1}'(s_i(\Lambda+Y)^2)={p_Y}'(s_i(\Lambda+Y)^2)$ 
if the multiplicity of $s_i(\Lambda+Y)^2$ is larger than 1 (see Section \ref{interpolation}). 
Then 
\begin{align*}
\lefteqn{[\Lambda,\varphi_1(|\Lambda+Y|)]} \\
 &=\sum_{k=1}^{n-1}a_k(Y)\sum_{l=0}^{k-1}|\Lambda+Y|^{2l}\left([\Lambda,Y^*](\Lambda+Y)+(\Lambda^*+Y^*)[\Lambda,Y]
 \right)|\Lambda+Y|^{2(k-1-l)} \\
 &=\sum_{k=1}^{n-1}a_k(0)\sum_{l=0}^{k-1}|\Lambda|^{2l}\left([\Lambda,Y^*]\Lambda+\Lambda^*[\Lambda,Y]\right)|\Lambda|^{2(k-1-l)}
 +o(\|[\Lambda,Y]\|_2). 
\end{align*}
Thus for $i<j$ with $\lambda_i\neq \lambda_j$, we get 
\begin{align*}
\lefteqn{[\Lambda,\varphi_1(|\Lambda+Y|)]_{ij}} \\
&=\sum_{k=1}^{n-1}a_k(0)\sum_{l=0}^{k-1}|\lambda_i|^{2l}|\lambda_j|^{2(k-1-l)}\overline{\lambda_i}(\lambda_i-\lambda_j)y_{ij}
+o(\|[\Lambda,Y]\|_2)\\
 &=\sum_{k=1}^{n-1}a_k(0)(x^k)^{[1]}(|\lambda_i|^{2},|\lambda_j|^{2})\overline{\lambda_i}(\lambda_i-\lambda_j)y_{ij}+o(\|[\Lambda,Y]\|_2)\\
 &=\psi_1^{[1]}(|\lambda_i|^2,|\lambda_j|^2)\overline{\lambda_i}(\lambda_i-\lambda_j)y_{ij}+o(\|[\Lambda,Y]\|_2)\\
  &=\frac{\varphi_1^{[1]}(|\lambda_i|,|\lambda_j|)}{|\lambda_i|+|\lambda_j|}\overline{\lambda_i}(\lambda_i-\lambda_j)y_{ij}
  +o(\|[\Lambda,Y]\|_2). 
\end{align*}

We can estimate $[\Lambda,\varphi_2(|\Lambda^*+Y^*|)]_{ij}$ in a similar way.  
\end{proof}

When $a=0$ and $T$ has an eigenvalue 0, the above result may not be very satisfactory: 
$\varphi_i(x)=x^\alpha$ with $\alpha/2\notin \N$ satisfies the assumption only if $\alpha\geq 2(n-1)$. 
In the next theorem, we relax the regularity at 0 while imposing analyticity.

For $r>0$ and $\varepsilon>0$, we set 
\[R_{r,\varepsilon}=\{z\in \C;\; 0\leq \Re z\leq r+\varepsilon,\; -\varepsilon \leq \Im z\leq \varepsilon\}.\]

\begin{theorem}\label{holomorphic} Let the notation be as above and assume that $\varphi$ satisfies the condition (C1) 
with $a=0$ and $\varphi_i(0)=0$. 
We further assume that $\varphi_i(\sqrt{x})$, $i=1,2$, 
extend to $\psi_i\in C(R_{b^2,\varepsilon})$ for some $\varepsilon>0$ such that they are holomorphic 
on the interior of $R_{b^2,\varepsilon}$.  
Furthermore, we assume that one of the following two conditions holds: 
\begin{itemize}
\item[$(1)$] The following integral converges for $i=1,2$: 
\[\int_{-\varepsilon}^{\varepsilon}\frac{|\psi_i(it)|}{t^2}dt<\infty.\] 
\item[$(2)$] $\varphi_2=0$, $\varphi_1'(x)/x$ is bounded on $(0,b]$, and 
\[\int_{-\varepsilon}^{\varepsilon}\frac{|\psi_1(it)|}{|t|}dt<\infty.\] 
\end{itemize}
Then the flow $\{F^\varphi_t(T)\}_{t\geq 0}$ converges to a normal matrix 
as $t$ tends to $\infty$ for any $T\in \M_n[0,b]$.  
\end{theorem}

\begin{proof} We verify Eq.(\ref{criterion}). 
Assume $1\leq i<j\leq n$ and $\lambda_i\neq \lambda_j$. 
Note that thanks to the integrability assumption and $\psi_i(0)=0$, the holomorphic functional calculus 
	\[\varphi(\Lambda+Y)=\frac{1}{2\pi i}\int_{\partial R_{b^2,\varepsilon}}
	\left(\frac{\psi_1(z)}{zI-|\Lambda+Y|^2}-\frac{\psi_2(z)}{zI-|\Lambda^*+Y^*|^2}\right)
	dz\]
makes sense. 
Thus 
\begin{align*}[\Lambda,P(Y)]_{ij}
	& =\frac{1}{2\pi i}\int_{\partial R_{b^2,\varepsilon}} \inpr{[\Lambda,\frac{1}{\zeta I-|\Lambda+Y|^2}]e_j}{e_i}\psi_1(\zeta)d\zeta \\
	&-\frac{1}{2\pi i}\int_{\partial R_{b^2,\varepsilon}} \inpr{[\Lambda,\frac{1}{\zeta I-|\Lambda^*+Y^*|^2}]e_j}{e_i}\psi_2(\zeta)d\zeta.
\end{align*}

The first integral is 
\begin{align*}
	&\frac{1}{2\pi i}\int_{\partial R_{b^2,\varepsilon}} \inpr{
		\frac{1}{\zeta I-|\Lambda+Y|^2}
		[\Lambda,|\Lambda+Y|^2] 
		\frac{1}{\zeta I-|\Lambda+Y|^2} e_j}{e_i}
	\psi_1(\zeta)d\zeta \\
	&=\frac{1}{2\pi i}\int_{\partial R_{b^2,\varepsilon}} \inpr{
		\frac{1}{\zeta I-|\Lambda+Y|^2}
		(\Lambda^*[\Lambda,Y]+[\Lambda,Y^*]\Lambda) 
		\frac{1}{\zeta I-|\Lambda+Y|^2} e_j}{e_i}
	\psi_1(\zeta)d\zeta \\
	&+\frac{1}{2\pi i}\int_{\partial R_{b^2,\varepsilon}} \inpr{
		\frac{1}{\zeta I-|\Lambda+Y|^2}
		(Y^*[\Lambda,Y]+[\Lambda,Y^*]Y) 
		\frac{1}{\zeta I-|\Lambda+Y|^2} e_j}{e_i}
	\psi_1(\zeta)d\zeta, 
\end{align*}
which can be computed from 
\begin{align*}
	&\frac{1}{2\pi i}\int_{\partial R_{b^2,\varepsilon}} \inpr{
		\frac{1}{\zeta I-|\Lambda+Y|^2}
		E_{kl} 
		\frac{1}{\zeta I-|\Lambda+Y|^2} e_j}{e_i}	\psi_1(\zeta)d\zeta\\ \nonumber
	&=\frac{1}{2\pi i}\int_{\partial R_{b^2,\varepsilon}} \inpr{\frac{1}{\zeta I-|\Lambda+Y|^2} e_k}{e_i}\inpr{\frac{1}{\zeta I-|\Lambda+Y|^2} e_j}{e_l}	\psi_1(\zeta)d\zeta. \nonumber
\end{align*}
Thus all we have to verify is  
\begin{align}\label{limit} 
	&\lim_{Y\to 0}
	\frac{1}{2\pi i}\int_{\partial R_{b^2,\varepsilon}} \inpr{\frac{1}{\zeta I-|\Lambda+Y|^2} e_k}{e_i}\inpr{\frac{1}{\zeta I-|\Lambda+Y|^2} e_j}{e_l}	\psi_1(\zeta)d\zeta\\ \nonumber
	&=\delta_{k,i}\delta_{l,j}\frac{1}{2\pi i}\int_{\partial R_{b^2,\varepsilon}}\frac{\psi_1(\zeta)}{(\zeta-|\lambda_i|^2)(\zeta-|\lambda_j|^2)}d\zeta,
\end{align}
and its counterpart for $|\Lambda^*+Y^*|$ as the rest of the computation is routine work. 
Note that we have 
\[\frac{1}{2\pi i}\int_{\partial R_{b^2,\varepsilon}}\frac{\psi_1(\zeta)}{(\zeta-|\lambda_i|^2)(\zeta-|\lambda_j|^2)}d\zeta
=\psi^{[1]}_1(|\lambda_i|^2,|\lambda_j|^2)
=\frac{\varphi_1^{[1]}(|\lambda_i|,|\lambda_j|)}{|\lambda_i|+|\lambda_j|}.\]

In the case (1), the Lebesgue convergence theorem verifies Eq.(\ref{limit}). 

Assume (2) now. 
Remark \ref{diagonalizable},(2) shows $y(t)_{pq}=0$ for all $1\leq p<q\leq m^g_1$, $t\geq 0$. 
Thus it suffices to verify  Eq.(\ref{limit}) under the additional condition $y_{pq}=0$ 
for all $1\leq p<q\leq m^g_1$, which we assume in the following argument.  
Then $\ker|\Lambda+Y(t)|=\mathrm{span}\{e_p\}_{p=1}^{m^g_1}$, which implies $\varphi(\Lambda+Y)_{ij}=0$ 
for $i\leq m^g_1$. 
We assume $m^g_1< i$ now. 
Since the rank of $\Lambda+Y$ is $n-m^g_1$, we have $s_p(\Lambda+Y)>0$ for $1\leq p\leq n-m^g_1$, 
and $s_p(\Lambda+Y)=0$ for $p\geq n-m^g_1+1$. 
The continuity of the singular numbers implies $s_q(\Lambda+Y)\to |\lambda_{n+1-q}|$ as $Y$ tends to 0. 
Thus there exist $\delta>0$ and $0<c_1<c_2<b$ such that for all $Y$ with $\|Y\|_2<\delta$, we have 
$s_p(\Lambda+Y)>c_2$ for $p\leq n-m^a_1$ and $s_p(\Lambda+Y)<c_1$ for $p>n-m^a_1$. 
We assume $\|Y\|_2<\delta$ in the following argument. 

Let $Q(Y)=\chi_{[0,c_1]}(|\Lambda+Y|)$, where $\chi_{[0,c_1]}$ is the characteristic function of $[0,c_1]$. 
Since we can replace $\chi_{[0,c_1]}$ with a Lipschitz function with the same values on $[0,c_1]$ and $[c_2,b]$, 
there exists $L>0$ satisfying $\|Q(Y)-Q(0)\|\leq L\|Y\|$ for all $Y$ with $\|Y\|_2<\delta$. 
Since 
\[|\inpr{\frac{1}{it I-|\Lambda+Y|^2} e_k}{e_i}\inpr{\frac{1}{it I-|\Lambda+Y|^2}(I-Q(Y)) e_j}{e_l}\psi_1(it)|
\leq \frac{|\psi_1(it)|}{|t|\sqrt{t^2+c_2^4}}, \]
the Lebesgue convergence theorem implies  
\begin{align*}
	&\lim_{Y\to 0}
	\frac{1}{2\pi i}\int_{\partial R_{b^2,\varepsilon}} \inpr{\frac{1}{\zeta I-|\Lambda+Y|^2} e_k}{e_i}\inpr{\frac{1}{\zeta I-|\Lambda+Y|^2}(I-Q(Y)) e_j}{e_l}	\psi_1(\zeta)d\zeta\\ \nonumber
	&=\delta_{k,i}\delta_{l,j}\frac{1}{2\pi i}\int_{\partial R_{b^2,\varepsilon}}\frac{\psi_1(\zeta)}{(\zeta-|\lambda_i|^2)(\zeta-|\lambda_j|^2)}d\zeta. 
\end{align*}
Thus it suffices to show 
\[\lim_{Y\to 0}
	\frac{1}{2\pi i}\int_{\partial R_{b^2,\varepsilon}} \inpr{\frac{1}{\zeta I-|\Lambda+Y|^2} e_k}{e_i}\inpr{\frac{1}{\zeta I-|\Lambda+Y|^2}Q(Y)e_j}{e_l}	\psi_1(\zeta)d\zeta=0\]
	
We choose an orthonormal basis $\{v_p\}_{p=1}^n$ of $\C^n$ satisfying $|\Lambda+Y|v_p=s_p(\Lambda+Y)v_p$. 
Since 
\begin{align*}
	&\inpr{\frac{1}{\zeta I-|\Lambda+Y|^2} e_k}{e_i}\inpr{\frac{1}{\zeta I-|\Lambda+Y|^2}Q(Y)e_j}{e_l} \\
	&=\sum_{p,q}\inpr{e_k}{v_p}\inpr{v_p}{e_i}
	\inpr{Q(Y)e_j}{v_q}\inpr{v_q}{e_l}\frac{1}{\zeta-s_p(\Lambda+Y)^2}\frac{1}{\zeta-s_q(\Lambda+Y)^2},
\end{align*}
where $p$ runs from 1 to $n-m^g_1$ and $q$ runs from $n-m^a_1+1$ to $n-m^g_1$, we have 
\begin{align*}
	 &\frac{1}{2\pi i}\int_{\partial R_{b^2,\varepsilon}} \inpr{\frac{1}{\zeta I-|\Lambda+Y|^2} e_k}{e_i}\inpr{\frac{1}{\zeta I-|\Lambda+Y|^2}Q(Y)e_j}{e_l}	\psi_1(\zeta)d\zeta  \\
	&= \sum_{p,q}\inpr{e_k}{v_p}\inpr{v_p}{e_i}
	\inpr{Q(Y)e_j}{v_q}\inpr{v_q}{e_l}\psi_1^{[1]}(s_p(\Lambda+Y)^2,s_q(\Lambda+Y)^2).
\end{align*} 
Since 
\[\|Q(Y)e_j\|_2=\|(Q(Y)-Q(0))e_j\|_2\leq L\|Y\|,\] 
and $|\psi_1^{[1]}(s_p(\Lambda+Y)^2,s_q(\Lambda+Y)^2)|\leq \|\psi_1'\|_\infty<\infty$, we get 
\begin{align*}
	&|\frac{1}{2\pi i}\int_{\partial R_{b^2,\varepsilon}} \inpr{\frac{1}{\zeta I-|\Lambda+Y|^2} e_k}{e_i}\inpr{\frac{1}{\zeta I-|\Lambda+Y|^2}Q(Y)e_j}{e_l}	\psi_1(\zeta)d\zeta|  \\
	&\leq  n^2L \|\psi_1'\|_\infty \|Y\|_2,
\end{align*} 
which finishes the proof. 
\end{proof}

\begin{example} Let $\varphi_1(0)=0$ and $\varphi_1(x)=x^2(\log(2b/x))^{-\alpha}$ with $\alpha>0$ for $x\in (0,b]$, and  let $\varphi_2=0$. 
Then $\varphi=(\varphi_1,\varphi_2)$ satisfies the condition (2), while it satisfies the condition (1) 
only if $\alpha>1$. 
\end{example}

The above theorem is still not applicable to $\varphi(X)=|X|$. 
We give a criterion applicable to it though we need an extra condition $m^a_1\leq m^g_1+1$. 

\begin{theorem}\label{noninvertible} Let the notation be as above and assume $T\neq 0$ and that $\varphi$ satisfies the condition (C1) 
with $a=0$ and $\lambda_1=0$. 
We further assume one of the following:
\begin{itemize}
	\item [$(1)$] Assume $m^a_1=m^g_1$ for $T\in \M_n[0,b]$ and $\varphi_k\in C^{n-1-m^g_1}[0,b]$ for $k=1,2$. 
	\item [$(2)$] Assume $m^a_1=m^g_1+1$ for $T\in \M_n[0,b]$, $\varphi_2=0$, and $\varphi_1\in C^{n-1-m^g_1}[0,b]$. 
\end{itemize}
Then the flow $\{F^\varphi_t(T)\}_{t\geq 0}$ converges to a normal matrix 
as $t$ tends to $\infty$.  
\end{theorem}

\begin{proof} We verify Eq.(\ref{criterion}). 
Let $1\leq i<j\leq n$ with $\lambda_i\neq \lambda_j$. 
	
We set $\psi_k(x)=\varphi_k(\sqrt{x})-\varphi_k(0)$ for $k=1,2$.   
Let $A(Y)=(\Lambda+Y)^*(\Lambda+Y)$ and $B(Y)=(\Lambda+Y)(\Lambda+Y)^*$. 
Then $A(Y)$ and $B(Y)$ are positive matrices of rank $n-m^g_1$ with eigenvalues $\{s_p(\Lambda+Y)^2\}_{p=1}^n$. 
(Recall that we assume $y_{st}=0$ for $1\leq s<t\leq m^g_1$.) 
Let $e(A(Y))$ (respectively $e(B(Y))$) be the support projections of $A(Y)$ (respectively $B(Y)$); 
that is, the projection onto the orthogonal complement of $\ker(\Lambda+Y)$ (respectively $\ker(\Lambda+Y)^*$). 

We try to apply the same argument as in the proof of Theorem \ref{Cn-1} by using the interpolation polynomial of degree 
at most $n-1-m^g_1$. 
Let 
\[q_{k,Y}(t)=\sum_{p=0}^{n-1-m^g_1}b_{k,p}(Y)t^p,\]
be the Hermite interpolation polynomial for $\psi_k$ at $\{s_p(\Lambda+Y)^2\}_{p=1}^{n-m^g_1}$ 
(see Section \ref{interpolation}).  
There are two extra points we need to take care of:   
To ensure $\psi_1(A(Y))=q_{1,Y}(A(Y))$ and $\psi_2(B(Y))=q_{2,Y}(B(Y))$, 
we should interpret $A(Y)^0$ and $B(Y)^0$ as $e(A(Y))$ and $e(B(Y))$, respectively. 
We use this convention in the following argument. 
Another point is that the equality $q_{k,Y}(0)=\psi_k(0)$ does not hold.  
The first point matters in that we have $[\Lambda,e(B(Y))]\neq 0$ in general  
(note that $[\Lambda,e(A(Y))]=0$ holds since $\ker (\Lambda+Y)=\mathrm{span}\{e_i\}_{i=1}^{m^g_1}$). 
The second point matters only when we estimate $[\Lambda,\varphi(\Lambda+Y)]_{ij}$ for $1\leq i\leq m^g_1$. 

(1) Since $\lim_{Y\to 0}s_p(\Lambda+Y)=|\lambda_{n+1-p}|>0$ for $1\leq p\leq n-m^g_1$, there exist $\delta>0$ and $0<c<b$ such that 
$\|Y\|_2<\delta$ assures $s_p(\Lambda+Y)>c$ for $1\leq p\leq n-m^g_1$. 
Note that we have $\psi_k\in C^{n-1-m^g_1}([c^2,b^2])$. 
Thus the same argument as in the proof of Theorem \ref{Cn-1} works provided that we can take care of the above two points. 
When $1\leq i\leq m^g_1$, we have $\psi_1(A(Y))_{ij}=0$ and there is no problem with the estimate of 
$[\Lambda,\psi_1(A(Y))]_{ij}$. 

Taking an appropriate contour $C$ surrounding $[c^2,b^2]$ but not 0, we can express $e(B(Y))$ as 
\[e(B(Y))=\frac{1}{2\pi i}\int_C\frac{d\zeta}{\zeta I-B(Y)},\]
and 
\[[\Lambda,e(B(Y))]=\frac{1}{2\pi i}\int_C\frac{1}{\zeta I-B(Y)}[\Lambda,B(Y)]\frac{1}{\zeta I-B(Y)}d\zeta.\]
Since we have 
\begin{align*}
 & \lim_{Y\to 0}\frac{1}{2\pi i}\int_C\inpr{\frac{1}{\zeta I-B(Y)}E_{kl}\frac{1}{\zeta I-B(Y)}e_j}{e_i}d\zeta\\
 &=\delta_{k,i}\delta_{l,j}\frac{1}{2\pi i}\int_C\frac{1}{(\zeta-|\lambda_i|^2)(\zeta-|\lambda_j|^2)}d\zeta \\
 &=\left\{
\begin{array}{ll}
\delta_{k,i}\delta_{l,j}\frac{1}{|\lambda_j|^2} , &\quad 1\leq i\leq m^g_1,  \\
0 , &\quad m^g_1 < i,
\end{array}
\right.
\end{align*}
we get 
\[[\Lambda,e(B(Y))]_{ij}=\left\{
\begin{array}{ll}
	-y_{ij}+o(\|[\Lambda,Y]\|_2) , &\quad 1\leq i\leq m^g_1,  \\
	o(\|[\Lambda,Y]\|_2) , &\quad m^g_1 < i.
\end{array}
\right.\]
Thus for $m^g_1<i$, the same argument as in the proof of Theorem \ref{Cn-1} works. 
Assume $1\leq i\leq m^g_1$. 
Then 
\begin{align*}
	&[\Lambda,\psi_2(B(Y))]_{ij}  \\
	&=b_{2,0}(Y)[\Lambda,e(B(Y))]_{ij}+\sum_{p=1}^{n-1-m^g_1}b_{2,p}(Y)[\Lambda,B(Y)^p]_{ij}\\
	&=b_{2,0}(Y)[\Lambda,e(B(Y))]_{ij}+\sum_{p=1}^{n-1-m^g_1}\sum_{k=0}^{p-1}b_{2,p}(Y)(B(Y)^k[\Lambda,B(Y)]B(Y)^{p-1-k})_{ij}\\
	&=-b_{2,0}(0)y_{ij}+\sum_{p=1}^{n-1-m^g_1}\sum_{k=0}^{p-1}b_{2,p}(0)(|\Lambda|^{2k}[\Lambda,Y]\Lambda^* |\Lambda|^{2(p-1-k)})_{ij}+	o(\|[\Lambda,Y]\|_2).
\end{align*} 
Since $\lambda_i=0$, the only terms with $k=0$ survive, and we get 
\begin{align*}
[\Lambda,\psi_2(B(Y))]_{ij} 
&=-b_{2,0}(0)y_{ij}-\sum_{p=1}^{n-1-m^g_1}b_{2,p}(0)y_{ij}|\lambda_j|^{2p}+o(\|[\Lambda,Y]\|_2)\\
&=-\psi_2(|\lambda_j|^2)y_{ij}+o(\|[\Lambda,Y]\|_2)\\
&=\psi_2^{[1]}(|\lambda_i|^2,|\lambda_j|^2)(\lambda_i-\lambda_j)\overline{\lambda_j}y_{ij}+o(\|[\Lambda,Y]\|_2). 
\end{align*} 
Therefore Eq.(\ref{criterion}) is verified. 

(2) We have $\lim_{Y\to 0}s_p(\Lambda+Y)=|\lambda_{n+1-p}|>0$ for $1\leq p\leq n-m^g_1-1$ and  $\lim_{Y\to 0}s_{n-m^g_1}(\Lambda+Y)=0$ now. 
Since $\psi_1 \in C[0,b^2]\cap C^{n-1-m^g_1}(0,b^2]$ and $\psi_1(0)=0$,  
we can apply Lemma \ref{extension} to $\psi_1$, and the statement follows from the same argument as above.  
\end{proof}

\begin{remark} The reason we assume $\varphi_2=0$ in (2) is that we cannot estimate $[\Lambda, e(B(Y))]$ 
by using the holomorphic functional calculus because 
\[\lim_{Y\to 0}s_{n-m^g_1}(\Lambda+Y)=0.\]
As we see below, we can directly estimate $[\Lambda, e(B(Y))]$ when $n$ is small. 
\end{remark}

We summarize the situation in the 3-by-3 matrix case. 

\begin{proposition} Let the notation be as above with $n=3$ and assume that $\varphi$ satisfies the condition (C1). 
\begin{itemize}
	\item [$(1)$] Assume $a=0$ and $\ker T\neq \{0\}$.   
	Then the flow $\{F^\varphi_t(T)\}_{t\geq 0}$ converges to a normal matrix as $t$ tends to $\infty$.  	
	\item [$(2)$] Assume that $a>0$ and $\varphi_i\in C^2[a,b]$ for $i=1,2$.  
	Then the flow $\{F^\varphi_t(T)\}_{t\geq 0}$ converges to a normal matrix as $t$ tends to $\infty$.  
\end{itemize}
\end{proposition}

\begin{proof}
The only case not covered by our results so far is $m^g_1=1$ and $m^a_1=2$ with $\varphi_2\neq 0$. 
We show the convergence in this case. 
Assume that $\lambda_1=\lambda_2=0$, $\lambda_3\neq 0$, and the rank of $\Lambda+Y$ is 2. 
We use the notation in the proof of the previous theorem. 
Then there is no problem in estimating $[\Lambda,\psi_1(A(Y))]$, and 
we need to estimate $[\Lambda,\psi_2(B(Y))]_{13}$ and $[\Lambda,\psi_2(B(Y))]_{23}$. 

We denote $s_i(\Lambda+Y)^2=t_i(Y)$ for simplicity. 
Using Proposition \ref{symmetric formula}, we get 
\begin{align*}
	&\psi_2(B(Y))  \\
	&= \frac{\psi_2(t_1(Y))-\psi_2(t_2(Y))}{t_1(Y)-t_2(Y)}B(Y)+\frac{t_1(Y)\psi_2(t_2(Y))-t_2(Y)\psi_2(t_1(Y))}{t_1(Y)-t_2(Y)}e(B(Y)).
\end{align*}
Since $\ker(\Lambda+Y)^*$ is spanned by $(0,\overline{\lambda_3},-\overline{y_{23}})^T$, we have
\[e(B(Y))=I-\frac{1}{|\lambda_3|^2+|y_{23}|^2}\begin{pmatrix}
	0 & 0 & 0 \\
	0 & |\lambda_3|^2 & -\overline{\lambda_3}y_{23} \\
	0 & -\lambda_3\overline{y_{23}} & |y_{23}|^2
\end{pmatrix}.\]
and 
\[[\Lambda,e(B(Y))]=
\frac{1}{|\lambda_3|^2+|y_{23}|^2}\begin{pmatrix}
	0 & 0 & 0 \\
	0 & 0 & -|\lambda_3|^2y_{23} \\
	0 & \lambda_3^2\overline{y_{23}} & 0
\end{pmatrix},\]
which shows $[\Lambda,e(B(Y))]=O(y_{23})$.
Since 
\[\lim_{Y\to 0}\frac{t_1(Y)\psi_2(t_2(Y))-t_2(Y)\psi_2(t_1(Y))}{t_1(Y)-t_2(Y)}=0,\]
we get 
\begin{align*}
	[\Lambda,\psi_2(B(Y))]_{ij}&=\psi_2^{[1]}(t_1(Y),t_2(Y))[\Lambda,B(Y)]_{ij}+o(\|[\Lambda,Y]\|_2) \\
	&=\frac{\varphi_2^{[1]}(|\lambda_3|,0)}{|\lambda_3|}(\lambda_i-\lambda_j)\overline{\lambda_j}y_{ij}+o(\|[\Lambda,Y]\|_2).  
\end{align*}
\end{proof}

We still do not know if the convergence holds in general for $\varphi(X)=|X|$.   
The first test case is $n=4$, $m^a_1=3$, and $m^g_1=1$. 
In this case, we get the following formula from Proposition \ref{symmetric formula}:
\begin{align*}
\lefteqn{[\Lambda,|\Lambda+Y|]} \\
 &=-\frac{1}{(s_1+s_2)(s_2+s_3)(s_3+s_1)}[\Lambda,A(Y)^2-\Tr(A(Y))A(Y)]\\
 &+\frac{s_1s_2+s_2s_3+s_3s_1}{(s_1+s_2)(s_2+s_3)(s_3+s_1)}[\Lambda,A(Y)],
\end{align*}
where $s_i=s_i(\Lambda+Y)$ and $A(Y)=(\Lambda+Y)^*(\Lambda+Y)$ with $\Lambda=\diag(0,0,0,\lambda_4)$. 
Since 
\[\lim_{Y\to 0}s_2(\Lambda+Y)=\lim_{Y\to 0}s_3(\Lambda+Y)=0,\]
we encounter difficulties in estimating the first term. 
This example suggests that for further analysis, we need a lower estimate of the singular numbers 
$s_i(\Lambda+Y)$ that approach 0 as $Y$ tends to 0.

\section{The case of Hilbert space operators}
Throughout this section, we assume that $H$ is a separable infinite dimensional Hilbert space. 
We denote by $\B(H)_{\mathrm{sa}}$ the set of self-adjoint operators in $\B(H)$. 
For $0\leq a<b$, we let 
\[\B(H)_{\mathrm{sa}}[a,b]=\{A\in \B(H)_{\mathrm{sa}};\; \sigma(A)\subset [a,b]\},\]
\[L\B(H)[a,b]=\{T\in \B(H);\; \sigma(|T|)\subset [a,b]\},\]
\[R\B(H)[a,b]=\{T\in \B(H);\; \sigma(|T^*|)\subset [a,b]\}.\]

We first generalize Lemma \ref{decreasing} by using the Dini derivatives. 
For $X\in \B(H)$, we set 
\[m(X)=\inf_{\|x\|=1}\|Xx\|=\min\sigma(|X|).\]
Note that $X$ is left invertible if and only if $m(X)>0$.  
 
\begin{lemma}\label{decreasing2} Let $\varphi=(\varphi_1,\varphi_2)$ be a pair of functions satisfying the condition (C0), and 
let $T\in L\B(H)[a,b]\cap R\B(H)[a,b]$. 
Let $X(t)$ be a solution of Eq.(\ref{ODE}) in $\B(H)$ defined on $t\in [0,\tau]$, where the derivative 
is taken with respect to the operator norm. 
Then $\|X(t)\|$ is decreasing with respect to $t$ on $[0,\tau]$. 
If moreover  $a>0$, then $m(X(t))$ is increasing.     
\end{lemma}

\begin{proof} Note that we implicitly assume that $X(t)\in L\B(H)[a,b]\cap R\B(H)[a,b]$ 
for all $t\in [0,\tau]$.  
Let $0\leq s< t\leq\tau$, and let $x\in H$ be a unit vector. 
Then since 
\begin{align*}
	\|X(s)x\|^2 &= \|X(t)x-(X(t)x-X(s)x)\|^2\\
	&= \|X(t)x\|^2-2\Re \inpr{(X(t)-X(s))x}{X(t)x}+\|(X(t)-X(s))x\|^2, 
\end{align*}
we get 
\begin{align*}
\|X(t)x\|^2	 &\leq \|X(s)\|^2+2\Re\int_s^t\inpr{[\varphi(X(r)),X(r)]x}{X(t)x}dr \\
	 &\leq \|X(s)\|^2+2(t-s)\Re \inpr{[\varphi(X(t)),X(t)]x}{X(t)x}\\
	 &+2\|X(t)\| \int_s^t\|[\varphi(X(r)),X(r)]-[\varphi(X(t)),X(t)]\|dr.
\end{align*}
Since $\|X(t)\|\in \sigma(|X(t)|)$, there exists a sequence of unit vectors $x_n\in H$ such that  
$\{|X(t)|x_n-\|X(t)\|x_n\}_n$ converges to 0. 
Then $\{\|X(t)x_n\|\}_n$ converges to $\|X(t)\|$. 
Let $X(t)=V(t)|X(t)|$ be the polar decomposition. 
Since 
\begin{align*}
\lefteqn{\inpr{[\varphi(X(t)),X(t)]x_n}{X(t)x_n}}	 \\
&=\inpr{\varphi_1(|X(t)|)V(t)|X(t)|x_n}{V(t)|X(t)|x_n}-\inpr{|X(t)|^2\varphi_1(|X(t)|)x_n}{x_n}\\
&+\inpr{|X(t)|^2\varphi_2(|X(t)^*|)x_n}{x_n}-\inpr{|X(t)|\varphi_2(|X(t)|)|X(t)|x_n}{x_n}, 
\end{align*}
we get 
\begin{align*}
	\limsup_{n\to\infty}\Re\lefteqn{\inpr{[\varphi(X(t)),X(t)]x_n}{X(t)x_n}}\\
	&\leq\|X(t)\|^2\limsup_{n\to\infty}(\inpr{\varphi_1(|X(t)|)V(t)x_n}{V(t)x_n}-\varphi_1(\|X(t)\|))\\
	&+\|X(t)\|^2\limsup_{n\to\infty}(\inpr{\varphi_2(|X(t)^*|)x_n}{x_n}-\varphi_2(\|X(t)\|))\\
	&\leq 0.
\end{align*}
Thus 
\[\|X(t)\|^2\leq \|X(s)\|^2+2\|X(t)\| \int_s^t\|[\varphi(X(r)),X(r)]-[\varphi(X(t)),X(t)]\|dr\]
and 
\[\frac{\|X(t)\|^2-\|X(s)\|^2}{t-s}\leq \frac{2\|X(t)\|}{t-s}\int_s^t\|[\varphi(X(r)),X(r)]-[\varphi(X(t)),X(t)]\|dr.\]
This shows 
\[\limsup_{s\uparrow t}\frac{\|X(t)\|^2-\|X(s)\|^2}{t-s}\leq 0,\]
\[\limsup_{t\downarrow s}\frac{\|X(t)\|^2-\|X(s)\|^2}{t-s}\leq 0,\]
and \cite[Chapter 4, Theorem 1.2]{Br1994} implies that $\|X(t)\|^2$ is a decreasing function on $[0,\tau]$. 

Now assume $a>0$. 
Then $X(t)$ is invertible for all $t\in [0,\tau]$. 
For $s,t$ and $x$ as above, we have 
 \begin{align*}
 	\|X(t)x\|^2	 &= \|X(s)x\|^2+2\Re\int_s^t\inpr{[\varphi(X(r)),X(r)]x}{X(t)x}dr-\|(X(t)-X(s))x\|^2 \\
 	&\geq m(X(s))^2+2(t-s)\Re \inpr{[\varphi(X(t)),X(t)]x}{X(t)x}\\
 	&-2\|X(t)\| \int_s^t\|[\varphi(X(r)),X(r)]-[\varphi(X(t)),X(t)]\|dr-\|X(t)-X(s)\|^2.
 \end{align*}
We take a sequence $\{y_n\}_n$ of unit vectors in $H$ such that $\{|X(t)|y_n-m(X(t))y_n\}_n$ converges to 0. 
Then since $m(X(t))=\min\sigma(|X(t)|)=\min\sigma(|X(t)^*|)$, we get 
\begin{align*}
	\liminf_{n\to\infty}\Re\lefteqn{\inpr{[\varphi(X(t)),X(t)]y_n}{X(t)y_n}}\\
	&\geq m(X(t))^2\liminf_{n\to\infty}(\inpr{\varphi_1(|X(t)|)V(t)y_n}{V(t)y_n}-\varphi_1(m(X(t))))\\
	&+m(X(t))^2\liminf_{n\to\infty}(\inpr{\varphi_2(|X(t)^*|)y_n}{y_n}-\varphi_2(m(X(t))))\\
	&\geq 0.
\end{align*}
and  
\begin{align*}
m(X(t))^2&\geq m(X(s))^2-2\|X(t)\| \int_s^t\|[\varphi(X(r)),X(r)]-[\varphi(X(t)),X(t)]\|dr\\
&-\|X(t)-X(s)\|^2.
\end{align*}
Thus 
\[\liminf_{s\uparrow t}\frac{m(X(t))^2-m(X(s))^2}{t-s}\geq 0,\]
\[\liminf_{t\downarrow s}\frac{m(X(t))^2-m(X(s))^2}{t-s}\geq 0,\]
and $m(X(t))^2$ is increasing on $[0,\tau]$. 
\end{proof}

\begin{remark} The above argument for $m(X(t))$ breaks down if $a=0$ and $\sigma(|X(t)|)\neq \sigma(|X(t)^*|)$.   
\end{remark}

Since $\dim H=\infty$, we have only $\sigma(|T|)\setminus \{0\}=\sigma(|T^*|)\setminus \{0\}$ in general. 
Note that $\sigma(|T|)\neq \sigma(|T^*|)$ occurs only if either $T$ is left invertible but not right invertible or $T$ is right invertible but not left invertible. 
As we are interested in the Aluthge flow of a left invertible $T$, 
we should not miss such cases. 
Since the second case can be reduced to the first case by taking adjoint, 
we treat only the first case. 

We can prove the following lemma in the same way as in the proof of the previous lemma. 

\begin{lemma}\label{decreasing3}
Let $0<a<b$ and let $\varphi_1\in C[a,b]$ be an increasing function. 
Let $T\in L\B(H)[a,b]$ and let $X(t)$ be a solution of 
\begin{equation}\label{LODE}
\frac{dX(t)}{dt}=[\varphi_1(|X(t)|),X(t)],\quad X(0)=T,
\end{equation} 
in $\B(H)$ defined on $[0,\tau]$, where the derivative is taken with respect to the operator norm. 
Then $\|X(t)\|$ is decreasing and $m(X(t))$ is increasing with respect to $t$ on $[0,\tau]$. 
\end{lemma}

We denote the set of compact operators by $\K(H)$, and the set of finite rank operators by $\F(H)$.  
Following \cite[Eq.(2.7)]{GK1969}, we define the singular values of $T\in \B(H)$ by 
\[s_n(T)=\inf\{\|T-S\|;\; S\in \F(H),\; \operatorname{rank} S\leq n-1\}.\]

For the theory of normed ideals, we refer the reader to \cite{GK1969} and \cite{S2005}. 
A two-sided ideal $\cJ$ of $\B(H)$ is said to be a normed ideal if $\cJ$ is equipped with a norm 
$\|\cdot\|_\cJ$ such that $(\cJ, \|\cdot \|_\cJ)$ is a Banach space and the following hold: 
\begin{itemize}
	\item [(1)] $\|X\|_\cJ=\|X\|$ for every rank one operator $X\in \F(H)$.
	\item [(2)] $\|YXZ\|_\cJ\leq \|Y\| \|X\|_\cJ \|Z\|$ for all $X\in \cJ$ and $Y,Z\in \B(H)$. 
\end{itemize}
It is known that the second condition is equivalent to the following unitary invariance property: 
\begin{itemize}
	\item [(2)'] $\|UXV\|_\cJ= \|X\|_\cJ$ for all $X\in \cJ$ and unitaries $U,V\in \B(H)$. 
\end{itemize}
Although it is customary to exclude the case $\cJ=\B(H)$, $\|\cdot\|_\cJ=\|\cdot\|$, we include it in the following discussion. 
For a given normed ideal $\cJ$, we denote by $\cJ^{(0)}$ the closure of $\F(H)$ with respect to the norm $\|\cdot\|_\cJ$. 
For example, if $\cJ=\B(H)$, we have $\cJ^{(0)}=\K(H)$.

We denote by $\B(H)^{-1}$ the group of bounded invertible operators. 
For a normed ideal $\cJ$, we set
\[G_\cJ=\{X\in \B(H)^{-1};\; X-I\in \cJ\},\]
which is a Banach Lie group (see \cite{de la Harpe1972}). 
For $T\in \B(H)$, its similarity orbit with respect to $G_\cJ$ is defined by 
\[\cO_\cJ(T)=\{gTg^{-1};\; g\in G_\cJ\}.\]
When $\cJ=\B(H)$, we denote it simply by $\cO(T)$.   

For $1\leq p<\infty$, we let $\|T \|_p=(\sum_{n=1}^\infty s_n(T)^p)^{1/p}$, and we denote by $\cJ_p$ the Schatten ideal 
\[\cJ_p(H)=\{T\in \K(H);\; \|T\|_p<\infty\}.\]

Let $\cJ$ be a normed ideal. 
We would like to discuss the solution of Eq.(\ref{ODE}) satisfying $X(t)-T\in \cJ$ under 
the condition $[T,\varphi(T)]\in \cJ$. 
For this purpose, we need to set up a relevant Lipschitz condition.  
 
\begin{definition} Let $\cJ$ be a normed ideal, let $\cX$ be a subset of $\B(H)$, and let $\psi :\cX\to \B(H)$ be a map. 
We say that $\psi$ is \textit{$\cJ$-Lipschitz} if there exists $C>0$ such that whenever $X_1,X_2\in \cX$ and $X_1-X_2\in \cJ$, 
we have $\psi(X_1)-\psi(X_2)\in \cJ$ and $\|\psi(X_1)-\psi(X_2)\|_\cJ\leq C\|X_1-X_2\|_\cJ$. 

We say that $f\in C[a,b]$ is a \textit{$\cJ$-Lipschitz function} on $[a,b]$ if the map $A\mapsto f(A)$ is $\cJ$-Lipschitz on $\B(H)_{\mathrm{sa}}[a,b]$. 
If moreover $\cJ=\B(H)$, we say that $f$ is an \textit{operator Lipschitz function} on $[a,b]$. 
\end{definition}

\begin{remark} A few remarks are in order. 
\begin{itemize}
\item [(1)] Let $1<p<\infty$. 
The Potapov--Sukochev theorem \cite[Theorem 1]{PS2011} says that $f\in C[a,b]$ is $\cJ_p$-Lipschitz on $[a,b]$ if and only if $f\in \Lip[a,b]$. 
\item [(2)] Kato \cite{K1973} showed that the absolute value function $|x|$ is not operator Lipschitz on $[-b,b]$  
(see  \cite{K1992} for related topics). 
\item [(3)] 
Let $f\in C^1[a,b]$. 
If $f'$ extends to a function that is the Fourier transform of a complex measure, 
then $f$ is $\cJ$-Lipschitz on $[a,b]$ for all $\cJ$ (see \cite[Theorem 1.1.1]{AP2016}). 
For example, if $f'$ extends to a periodic function with absolutely convergent Fourier series, then 
$f$ satisfies this criterion. 
For example, if $\alpha>1$, the function $|x|^\alpha$ is $\cJ$-Lipschitz on $[-b,b]$ for all $\cJ$.   
\end{itemize}
\end{remark}

If $0<a$ and $f\in C[a,b]$ is $\cJ$-Lipschitz on $[a,b]$, the map $T\mapsto f(|T|)$ on 
$L\B(H)[a,b]$ is $\cJ$-Lipschitz. 
This follows from the fact that $f(|T|)=f((T^*T)^{1/2})$ and $x^{1/2}$ is smooth on $[a^2,b^2]$. 
This is no longer the case if $a=0$ (think of $|x|$). 
For $f\in C[0,b]$, let $\tilde{f}$ be the even extension $\tilde{f}(x)=f(|x|)$ of $f$ to the interval $[-b,b]$.
For $T\in \B(H)$ with $\|T\|\leq b$, we have 
\[\sigma\left(\begin{pmatrix}
	0 & T^* \\
	T & 0
\end{pmatrix}\right)\subset [-b,b],\]
and 
\[\tilde{f}\left(\begin{pmatrix}
	0 & T^* \\
	T & 0
\end{pmatrix}\right)=
\begin{pmatrix}
f(|T|) & 0 \\
0 & f(|T^*|)
\end{pmatrix}. 
\]
Thus the $\cJ$-Lipschitz property of $\tilde{f}$ on $[-b,b]$ ensures that 
the maps $T\mapsto f(|T|)$ and $T\mapsto f(|T^*|)$ are $\cJ$-Lipschitz on the ball $\{T\in \B(H);\; \|T\|\leq b\}$. 
For example, if $\alpha>1$, the map $T\mapsto |T|^\alpha$ on a bounded subset of $\B(H)$ is $\cJ$-Lipschitz for every $\cJ$.

\begin{theorem}\label{exi&uni2}
Let $\varphi=(\varphi_1,\varphi_2)$ be a pair of functions satisfying the condition (C0), and let $T\in \B(H)$. 
Let $\cJ$ be a normed ideal.  
We assume 
\begin{itemize} 
\item[$(1)$] if $0<a$, $\varphi_1$ and $\varphi_2$ are $\cJ$-Lipschitz functions on $[a,b]$, and 
$\sigma(|T|), \sigma(|T^*|)\subset (a,b)$.  
\item[$(2)$] if $a=0$, $\tilde{\varphi_1}$ and $\tilde{\varphi_2}$ are $\cJ$-Lipschitz functions on $[-b,b]$, 
and $\sigma(|T|), \sigma(|T^*|)\subset [0,b)$. 
\end{itemize}
Assume $[\varphi(T),T]\in \cJ$. 
Then there exists a unique global solution $X(t)$ of Eq.(\ref{ODE}) on $[0,\infty)$ such that 
$X(t)-T$ belongs to  $C^1([0,\infty),\cJ)$, where $\cJ$ is equipped with the norm $\|\cdot\|_\cJ$.  
The solution stays in the similarity orbit $\cO(T)$ of $T$. 
If moreover $\varphi(T)\in \cJ$, it stays in $\cO_\cJ(T)$. 
\end{theorem}

\begin{proof} Let $\cZ=\{Z\in \cJ;\; T+Z\in L\B(H)[a,b]\cap R\B(H)[a,b] \}$, which is a neighborhood 
of $0$ in the Banach space $\cJ$. 
We define $\psi:\cZ\to \B(H)$ by $\psi(Z)=[\varphi(T+Z),T+Z]$. 
Then
\[\psi(Z)=[\varphi(T+Z)-\varphi(T),T+Z]+[\varphi(T),Z]+[\varphi(T),T]\in \cJ,\]
and $\psi$ is $\cJ$-Lipschitz. 
Now  Eq.(\ref{ODE}) is equivalent to the $\cJ$-valued ODE
\[\frac{dZ(t)}{dt}=\psi(Z(t)),\quad Z(0)=0,\] 
satisfying the Lipschitz condition, which has a unique local solution. 
Now we can show the existence of a unique global solution as in the proof of Theorem \ref{exi&uni} by using 
Lemma \ref{decreasing2} instead of Lemma \ref{decreasing}.

We can show $X(t)\in \cO(T)$, as in the proof of Proposition \ref{orbit}. 
Note that $t\mapsto X(t)$ is norm continuous. 
Solving the linear ODEs, 
\[\frac{dV(t)}{dt}=-V(t)\varphi(X(t)),\quad V(0)=I,\]
\[\frac{dW(t)}{dt}=\varphi(X(t))W(t),\quad W(0)=I,\]
we can get norm continuous families $\{V(t)\}_{t\geq 0}$ and $\{W(t)\}_{t\geq 0}$ 
in $\B(H)$ satisfying $V(0)=W(0)=I$, $V(t)W(t)=I$, and $V(t)X(t)W(t)=T$. 
To show $X(t)\in \cO(T)$, it suffices that $V(t)$ is (left) invertible for all $t\geq 0$. 
Let $\cI=\{t\in [0,\infty);\; \text{$V(t)$ is invertible}\}$. 
Then $\cI$ is an open subset of $[0,\infty)$ containing $0$. 
Suppose $\cI\neq [0,\infty)$. 
Then the connected component of $0$ in $\cI$ is of the form $[0,t_1)$ with $t_1\notin \cI$. 
Since for every $t\in [0,t_1)$, we have $W(t)V(t)=I$, we get $W(t_1)V(t_1)=I$ by taking the norm limit. 
This shows that $V(t_1)$ is invertible, which is a contradiction. 
Therefore $V(t)$ is invertible for all $t\geq 0$. 

Assume $\varphi(T)\in \cJ$ now. 
Since $X(t)-T\in \cJ$ and $\varphi$ is $\cJ$-Lipschitz, we get $\varphi(X(t))\in \cJ$. 
Thus $\varphi(X(t))$ is a $\cJ$-valued continuous function with respect to the norm $\|\cdot \|_\cJ$. 
Let $\tilde{V}(t)=V(t)-I$. 
Then $\tilde{V}(t)$ is a solution of the $\cJ$-valued linear ODE
\[\frac{d\tilde{V}(t)}{dt}=-\tilde{V}(t)\varphi(X(t))-\varphi(X(t)),\quad \tilde{V}(0)=0,\]
and so $\tilde{V}(t)\in \cJ$. 
Thus $V(t)\in G_\cJ$, and $X(t)\in \cO_\cJ(T)$. 
\end{proof}

In a similar way, we can show the following.

\begin{theorem}\label{exi&uni3}
Let $0<a<b$, let $\varphi_1\in C[a,b]$ be a strictly increasing function, and  
assume that $T\in \B(H)$ satisfies $\sigma(|T|)\subset (a,b)$. 
Let $\cJ$ be a normed ideal such that $[\varphi_1(|T|),T]\in \cJ$ and $\varphi_1$ is $\cJ$-Lipschitz on $[a,b]$. 
Then there exists a unique global solution $X(t)$ of Eq.(\ref{LODE}) on $[0,\infty)$ such that 
$X(t)-T$ belongs to  $C^1([0,\infty),\cJ)$, where $\cJ$ is equipped with the norm $\|\cdot\|_\cJ$.   	
The solution stays in the similarity orbit $\cO(T)$ of $T$. 
If moreover $\varphi_1(|T|)\in \cJ$, it stays in $\cO_\cJ(T)$. 
\end{theorem}

Since it would be rather awkward to treat the two cases in Theorems \ref{exi&uni2} and \ref{exi&uni3} differently in notation, 
we treat the latter case as a special case of the former with $\varphi_2=0$ abusing the notation slightly.   
We define an operator flow $F^\varphi_t(T)$ by the unique solution of Eq.(\ref{ODE}) (or Eq.(\ref{LODE})) as in the matrix case. 
Strictly speaking, the flow may depend on the operator ideal $\cJ$, but we do not bother to put it in the notation. 
We define the Haagerup flow $\{F^H_t(T)\}_{t\geq 0}$ and the Aluthge flow $\{F^A_t(T)\}_{t\geq 0}$ 
as in the matrix case. 
Note that the latter is defined for every left invertible operator $T$. 

\begin{lemma}\label{necessary} Let $\varphi$, $T$, and $\cJ$ be either as in Theorem \ref{exi&uni2} or as in Theorem \ref{exi&uni3}. 
For the flow $\{F^\varphi_t(T)\}_t$ to converge with respect to the norm $\|\cdot \|_\cJ$, it is necessary that 
there exists $Z\in \cJ$ satisfying $[\varphi(T+Z),T+Z]=0$, that is, $T+Z$ is a fixed point of $F^\varphi_t$. 
\end{lemma}

\begin{proof}
Assume that $\{F^\varphi_t(T)\}_t$ converges to $X_\infty\in \B(H)$ as $t \to \infty$ in $\|\cdot\|_\cJ$, 
and hence also in the operator norm. 
Since the domain of $\varphi$ is norm-closed and $\varphi$ is $\cJ$-Lipschitz, the limit $X_\infty$ remains in the domain of 
$\varphi$ and $[\varphi(F^\varphi_t(T)),F^\varphi_t(T)]$ converges to $[\varphi(X_\infty),X_\infty]$ in $\|\cdot\|_\cJ$. 
Suppose $[\varphi(X_\infty),X_\infty]\neq 0$. 
Then there exist $x,y\in H$ satisfying $\inpr{[\varphi(X_\infty),X_\infty]x}{y}\neq 0$, and 
\[\inpr{F^\varphi_t(T)x}{y}=\inpr{Tx}{y}+\int_0^t \inpr{[\varphi(F^\varphi_r(T)),F^\varphi_r(T)]x}{y}dr,\]
holds. 
This implies that
\[\lim_{t\to\infty}\int_0^t \inpr{[\varphi(F^\varphi_r(T)),F^\varphi_r(T)]x}{y}dr\]
converges, which is a contradiction as $\inpr{[\varphi(F^\varphi_r(T)),F^\varphi_r(T)]x}{y}$ converges to a 
non-zero number as $r\to\infty$. 
Thus $[\varphi(X_\infty),X_\infty]=0$ holds. 
\end{proof}

This criterion immediately shows that there exists a non-convergent flow.   

\begin{example}\label{Haagerupshift}
The Haagerup flow $\{F^H_t(S)\}_{t\geq 0}$ starting from the unilateral shift $S$ never converges 
in the operator norm (and hence in any unitarily invariant norm). 
Indeed, since $[S^*,S]$ is a rank one projection, we can take $\K$ equipped with the operator norm as 
the operator ideal $\cJ$. 
Assume that there exists $Z\in \K$ satisfying 
\[ [[(S+Z)^*,S+Z],S+Z]=0.\] 
Then the Kleinecke--Shirokov theorem implies that $S+Z$ is normal (see \cite[Problem 232]{Hal1982}). 
However, this is a contradiction because the Fredholm index of $S+Z$ is $-1$. 
Therefore there is no such $Z\in \K$.  
\end{example}

Our favorite playground for the Aluthge flow is the class of weighted shifts,   
as was already observed for the Aluthge transform (see \cite{CJL2005}). 

\begin{example} Let $H=\ell^2$, and let $S$ be the unilateral shift on $H$. 
We regard $\ell^\infty$ as the diagonal subalgebra of $\B(H)$ by identifying $f\in \ell^\infty$ with the 
multiplication operator of $f$. 
Then every unilateral weighted shift is of the form $Sf$ for some $f\in \ell^\infty$. 
Moreover, up to unitary equivalence, we may and do assume $f\geq 0$.  
To define the Aluthge flow, we need to assume that there exists a constant $a>0$ satisfying $a\leq |T|$. 
Thus we assume $T=Sf$ with $a\leq f$ in what follows. 

We can explicitly describe the Aluthge flow $F^A_t(Sf)$ as follows. 
Let $B:\ell^\infty\to\ell^\infty$ be the backward shift. 
By the uniqueness of the solution to Eq.(\ref{ODE}), we can find the solution of the form $Sh(t)$, $h(t)\in \ell^\infty$, with 
\[\frac{dh}{dt}=(B-I_{\ell^\infty})(\log h)h,\quad h(0)=f,\]
or equivalently with 
\[\frac{d\log h}{dt}=(B-I_{\ell^\infty})(\log h),\quad h(0)=f.\]
Let $\{\cP_t\}_{t\geq 0}$ be
the semigroup given by $e^{t(B-I_{\ell^\infty})}$ acting on $\ell^\infty$. 
Then we can express the solution as  
\begin{equation}\label{Poisson}
F^A_t(Sf)=Se^{\cP_t(\log f)}.
\end{equation}

Assume that there exists a scalar $\lambda>0$ satisfying $f-\lambda\in c_0$. 
Then we can show that $\{F^A_t(Sf)\}_{t\geq 0}$ converges to $\lambda S$ in the operator norm. 
Indeed, let $g=\log (f/\lambda)$. 
Then we have $F^A_t(Sf)=\lambda Se^{\cP_t(g)}$.  
Note that $g\in c_0$, and $\{\|B^{n}g\|_\infty\}_n$ converges to 0. 
Since 
\[\cP_t(g)=\sum_{n=0}^\infty \frac{t^ne^{-t}}{n!}B^{n}g,\] 
which is the average of $\{B^ng\}$ with respect to the Poisson distribution with mean $t$, we get 
\[\|\cP_t(g)\|_\infty\leq \sum_{n=0}^\infty \frac{t^ne^{-t}}{n!}\|B^{n}g\|_\infty \to 0, \quad (t\to\infty).\]
Therefore 
\[\lim_{t\to \infty}\|1-e^{\cP_t(g)}\|_\infty=0.\]

Let $\cJ$ be an operator ideal satisfying $\cJ=\cJ^{(0)}$. 
In a similar way, we can show  the convergence of $\{F^A_t(Sf)\}_{t\geq 0}$ to $\lambda S$ in $\|\cdot\|_\cJ$ 
if $f-\lambda\in \cJ\cap \ell^\infty$.  
\end{example}

\begin{remark}
Let $0<\lambda<1$, and let $\Delta_\lambda$ be the $\lambda$-Aluthge transform. 
Then for positive invertible $f\in \ell^\infty$, we have $\Delta_\lambda^n(Sf)=Se^{\cP_{\lambda,n}(\log f)}$ with 
\[\cP_{\lambda,n}(g)=\sum_{k=0}^n \begin{pmatrix}
	n \\
	k
\end{pmatrix}
(1-\lambda)^{n-k}\lambda^k B^{k}g,\]
which is the average of $\{B^{k}g\}$ with respect to the binomial distribution 
(see \cite{CJL2005}, and also \cite{OY2025} for related topics.)   
In view of Proposition \ref{Aluthge}, we can see why the Poisson distribution appears in the case of the Aluthge flow   
via the law of small numbers. 
\end{remark}

\begin{proposition}[cf. {\cite[Corollary 3.3]{CJL2005}}] Let $g\in \ell^\infty$ be given by 
\[g_n=\begin{cases}
	2^{-2N}(n-2^{2N}),  &  2^{2N}\leq n < 2^{2N+1},\\
	2^{-2N-1}(2^{2N+2}-n),  &  2^{2N+1}\leq n< 2^{2N+2},
\end{cases} 
\]
where $N\in \{0\}\cup\N$, and let $T=Se^g$. 
Then 
\[[\log|T|,T] =S(Bg-g)e^g\in \cJ_p,\]
for all $p>1$, and $\{F^A_t(T)\}_{t\geq 0}$ does not converge in the weak operator topology as $t\to \infty$. 
\end{proposition}

\begin{proof} By definition, $0\leq g\leq 1$ and  
\[(Bg)_n-g_n=\begin{cases}
	2^{-2N},  &  2^{2N}\leq n < 2^{2N+1}, \\
	-2^{-2N-1},  &  2^{2N+1}\leq n< 2^{2N+2}.
\end{cases}  
\]
Thus $Bg-g\in \ell^p$ for all $p>1$, and  $[\log|T|,T]\in \cJ_p$ for all $p>1$. 
Since 
\[\inpr{F^A_t(T)\delta_n}{\delta_{n+1}}=\inpr{e^{\cP_t(g)}\delta_n}{\delta_n},\]
to prove non-convergence, it suffices to show that 
\[\left\{\sum_{k=0}^\infty \frac{t^ke^{-t}}{k!}g_{n+k}\right\}_{t\geq 0}\]
does not converge for some (in fact all) $n\in \N$ as $t\to\infty$. 

We fix $0<\varepsilon <1/4$. 
Let $X_t$ be a random variable whose law is the Poisson distribution with mean $t$. 
Since $\{(X_t-t)/\sqrt{t}\}_{t\geq 0}$ converges in distribution to the standard normal distribution 
as $t\to \infty$, we can choose $a>0$ and $t_0>0$ such that 
$\operatorname{Pr}(|X_t-t|\geq a\sqrt{t})<\varepsilon$ holds for all $t\geq t_0$.

In the following argument, we fix $n\in \N$.   
If $t=2^{2N}-n>t_0$ and $a/2^N\leq 1/8$, we have
\[2^{2N}(1-\frac{1}{8})\leq n+k\leq 2^{2N}(1+\frac{1}{8}),\]
for $|k-t|<a\sqrt{t}$, and $g_{n+k}\leq 1/4$ holds. 
Thus   
\[\sum_{k=0}^\infty \frac{t^ke^{-t}}{k!}g_{n+k}\leq \varepsilon+\frac{1}{4}.\]
Similarly, we can show that  if $t=2^{2N+1}-n>t_0$ with $a/2^{N+1/2}<1/8$, 
we have $g_{n+k}\geq 3/4$ for $|k-t|<a\sqrt{t}$, and 
\[\sum_{k=0}^\infty \frac{t^ke^{-t}}{k!}g_{n+k}\geq\frac{3}{4}(1-\varepsilon).\]
Since 
\[\varepsilon+\frac{1}{4}<\frac{3}{4}(1-\varepsilon),\]
this proves the statement. 
\end{proof}

\begin{definition}Thanks to Lemma \ref{decreasing2} and Lemma \ref{decreasing3}, the following limit exists: 
\[r^\varphi(T):=\lim_{t\to\infty}\|F^\varphi_t(T)\|.\]
We denote it by $r^H(T)$ and $r^A(T)$ in the case of the Haagerup flow and the Aluthge flow respectively. 
Since $F^\varphi_t(T)$ is in the similarity orbit $\cO(T)$, we have $r(T)\leq r^\varphi(T)$. 
Corollary \ref{numerical} shows that $r(T)=r^\varphi(T)$ holds when $T$ is a matrix. 
\end{definition}

\begin{question} Does $r(T)=r^\varphi(T)$ hold when $\varphi$ is operator Lipschitz? 
An affirmative answer to this question would provide a counterpart of Yamazaki's theorem \cite[Theorem 1]{Y2002}.    
\end{question}

We give an example, other than matrices, in favor of an affirmative answer. 

\begin{proposition} Let $T=Sf$ be a left invertible weighted shift, where $S$ is the unilateral shift and $f\in \ell^\infty$ 
is positive and invertible. 
Then  $r(T)=r^A(T)$ holds.  
\end{proposition}

\begin{proof} Since $r(sT)=sr(T)$ and $r^A(sT)=sr^A(T)$ for a scalar $s>0$, we may and do assume that $f\geq 1$ to show 
the statement.  
Let $g=\log f\geq 0$. 
For $n\in \N$, we set 
\[M_n=\frac{1}{n}\sum_{k=0}^{n-1}B^k.\]
Then
\begin{align*}
\log r(T) & =\lim_{n\to \infty}\frac{1}{n}\log \|T^n\|=\lim_{n\to \infty}\frac{1}{n}\log \|S^ne^{nM_ng}\|\\
	 & =\lim_{n\to \infty}\frac{1}{n}\log \|e^{nM_ng}\|_\infty=\lim_{n\to\infty}\|M_ng\|_\infty. 
\end{align*}
On the other hand, Eq.(\ref{Poisson}) shows 
\[\log r^A(T)=\lim_{t\to\infty}\|\cP_tg\|_\infty.
\]
Our goal is to show 
\[\lim_{t\to\infty}\|\cP_tg\|_\infty\leq \lim_{n\to\infty}\|M_ng\|_\infty.\] 

Let $p_t$ be the Poisson distribution with mean $t$, which we regard as an element in $\ell^1(\{0\}\cup \N)$.  
Then the Stirling formula shows $\|\delta_1*p_t-p_t\|_1=O(\frac{1}{\sqrt{t}})$ as $t\to\infty$, and so 
\[\|M_n\cP_t-\cP_t\|\leq \|\frac{1}{n}\sum_{k=0}^{n-1}\delta_k*p_t-p_t\|_1\to 0,\quad (t\to\infty).\]
Thus 
\[\lim_{t\to\infty}\|\cP_t g\|_\infty=\lim_{t\to\infty}\|M_n\cP_t g\|_\infty=\lim_{t\to\infty}\|\cP_t M_ng\|_\infty\
\leq \|M_ng\|_\infty,\]
which shows the desired inequality. 
\end{proof}

\begin{question} It is an intriguing problem to investigate potential generalizations of Corollary \ref{sca+nil} by replacing 
the class of nilpotent matrices with suitable subclasses of quasi-nilpotent operators, such as 
the class of quasi-nilpotent weighted shifts or the class of quasi-nilpotent compact operators (Volterra operators).   
\end{question}

\begin{question}
It is an interesting question to determine whether the existence of $Z$ as in Lemma \ref{necessary} is a sufficient condition for the convergence of the flow $\{F^\varphi_t(T)\}_{t\geq 0}$, under the additional assumption $\cJ=\cJ^{(0)}$ and an appropriate regularity condition on $\varphi$. 
Indeed, the above examples suggest that the case $\cJ=\B(H)$ should be excluded; hence $\cJ=\cJ^{(0)}$ is a reasonable additional condition to impose. 
The matrix case further suggests that requiring an appropriate regularity condition on $\varphi$ would render the question more tractable. 
This question may be interesting even for the Aluthge transform under an appropriate formulation. 
\end{question}

\begin{question} Let $M$ be a finite von Neumann algebra and let $T\in M$. 
Haagerup \cite[Section 2]{HS2009} showed that the Haagerup flow $\{F^H_t(T)\}_{t\geq 0}$ converges to a normal element 
in the $*$-distribution sense. 
Does this result remain valid  for more general flows $\{F^\varphi_t(T)\}_{t\geq 0}$? 
\end{question}
\section{Conclusion}
We have introduced a family of matrix (and operator) flows that unifies a continuous analogue
of the Aluthge transform and the Haagerup equation by using ordinary differential equations. 
Monotonicity of the operator norm yields global existence under the
Lipschitz condition. 

In the matrix case, separating the dynamics into upper triangular
and unitary parts, we obtain convergence to a normal matrix under the
regularity assumptions established in Section 3. 
For diagonalizable initial matrices, the convergence is exponential and the limit belongs
to the same similarity orbit. 
For general initial matrices, the main step is to control the motion of the unitary conjugate of the diagonal
part, even when the strictly upper triangular part decays slowly.

For operators on an infinite-dimensional Hilbert space, we establish
global existence in affine spaces associated with normed ideals under
appropriate Lipschitz assumptions. 
Examples show that the finite-dimensional convergence results do not extend without
additional hypotheses: the Haagerup flow of the unilateral shift fails
to converge in norm, and an Aluthge flow of a left invertible weighted
shift may fail to converge even in the weak operator topology.
For weighted shifts, the explicit description of the Aluthge flow
through a Poisson semigroup provides both convergence results and
counterexamples.

Several questions remain open. 
In the matrix case, it is natural to seek weaker regularity assumptions, particularly near the zero
singular value. 
In infinite dimensions, the main problems are to find useful sufficient conditions for convergence,  
and to determine whether the operator norm along the flow always tends to the spectral radius under suitable Lipschitz assumptions. 
\appendix
\section{Appendix}
\subsection{Aluthge flow}
\begin{proposition}\label{Aluthge}
Let $T\in \B(H)$ be a left invertible operator on a Hilbert space $H$. 
Then $\{\Delta_{t/n}^n(T)\}_n$ converges to the Aluthge flow $F^A_t(T)$ in the operator norm 
uniformly in $t$ on compact subsets of $[0,\infty)$. 
\end{proposition}

\begin{proof} Let $M=\|T\|=\max\sigma(|T|)$ and let $m=m(T)=\min\sigma(|T|)>0$. 
Then it is well known that $\sigma(|\Delta_\lambda(T)|)\subset [m,M]$ holds. 
Indeed, let $T=U|T|$ be the polar decomposition and let $x\in H$ be a unit vector. 
Then we have 
\begin{align*}
\|\Delta_\lambda(T)x\|^2 &=\inpr{|T|^{2\lambda}U|T|^{1-\lambda}x}{U|T|^{1-\lambda}x}\leq
 M^{2\lambda}\inpr{U|T|^{1-\lambda}x}{U|T|^{1-\lambda}x} \\
 &\leq  M^{2\lambda}\inpr{|T|^{1-\lambda}x}{|T|^{1-\lambda}x} \leq  M^2. 
\end{align*}
Similarly, we can show $\|\Delta_\lambda(T)x\|\geq m$, and so $\sigma(|\Delta_\lambda(T)|)\subset [m,M]$. 

We fix $\tau>0$ and let $0\leq t\leq \tau$. 
Let $L$ be the Lipschitz constant of $[\log |X|,X]$ on $L\B(H)[m,M]$. 
We take a sufficiently large $n\in \N$, and set $Y_k=\Delta_{t/n}^k(T)$, $Z_k=F^A_{kt/n}(T)$ for $k=0,1,2,\cdots,n$. 
Then 
\[Y_{k+1}=Y_k+\frac{t}{n}[\log |Y_k|,Y_k]+\varepsilon_k \] 
with 
\[\|\varepsilon_k\|=\left\|e^{\frac{t}{n}\log |Y_k|}Y_ke^{-\frac{t}{n}\log |Y_k|}-Y_k-\frac{t}{n}[\log |Y_k|,Y_k]\right\|\leq \frac{C_1t^2}{n^2},\]
where $C_1>0$ depends only on $m$,$M$, and $\tau$. 
We have 
\[Z_{k+1}=Z_k+\frac{t}{n}[\log|Z_k|,Z_k]+\delta_k\]
with 
\[\|\delta_k\|=\left\|\int_{kt/n}^{(k+1)t/n}([\log|F^A_s(T)|,F^A_s(T)]-[\log|Z_k|,Z_k])ds\right\|\leq\frac{C_2t^2}{n^2},\]
where $C_2>0$ depends only on $m$, $M$, $\tau$, and $L$. 
Thus 
\[\|Y_{k+1}-Z_{k+1}\|\leq \left(1+\frac{Lt}{n}\right)\|Y_k-Z_k\|+(C_1+C_2)\frac{t^2}{n^2},\]
and 
\[\|\Delta_{t/n}^n(T)-F^A_t(T)\|=\|Y_n-Z_n\|\leq \frac{(C_1+C_2)t(e^{Lt}-1)}{Ln}\leq \frac{(C_1+C_2)\tau (e^{L\tau}-1)}{Ln}.\]
\end{proof}

\subsection{Haagerup flow}
\begin{proposition}\label{gradient} The Haagerup flow on $M_n(\C)$ is the gradient flow with respect to 
	the energy function $E(X)=\|[X^*,X]\|_2^2/4=\Tr([X^*,X]^2)/4$. 
\end{proposition}

\begin{proof}
	Recall that the standard real inner product of $M_n(\C)$ is given by 
	$\inpr{X}{Y}=\Re\Tr(Y^*X)$.  	
	We have 
	\begin{align*}
		\frac{dE(X+tY)}{dt}{\bigg|}_{t=0} &=\frac{1}{2}\Tr([X^*,X]([Y^*,X]+[X^*,Y]))\\
		&=\Re \Tr([X,[X^*,X]]Y^*)\\
		&=-\inpr{[[X^*,X],X]}{Y},
	\end{align*}
	which shows that Eq.(\ref{HaagerupODE}) is the gradient equation for $E(X)$. 
\end{proof}

\subsection{Divided differences}
Let $f\in C^n[a,b]$. 
For $k=0,1,\cdots, n$, the $k$-th divided difference $f^{[k]}\in C([a,b]^{k+1})$ is inductively defined as follows. 
We set $f^{[0]}(x)=f(x)$. 
Assume that $f^{[k-1]}$ is defined for $k\leq n$. 
If there exist $1\leq i<j\leq k+1$ with $x_i\neq x_j$, we set $f^{[k]}$ to be
\begin{align*}
	 &f^{[k]}(x_1,x_2,\cdots,x_{k+1})  \\
	&= \frac{f^{[k-1]}(x_1,\cdots,x_{i-1},x_{i+1},\cdots, x_{k+1})-f^{[k-1]}(x_1,\cdots,x_{j-1},x_{j+1},\cdots,x_{k+1})}{x_j-x_i},
\end{align*}
and if $x_1=x_2=\cdots=x_{k+1}$, we set 
\[f^{[k]}(x_1,x_2,\cdots,x_{k+1}) =\frac{f^{(k)}(x_1)}{k!}.\]
It turns out that $f^{[k]}$ is well-defined, symmetric, and continuous on $[a,b]^{k+1}$. 
When $\{x_i\}_{i=1}^{n+1}$ are distinct, we have 
\[f^{[n]}(x_1,x_2,\cdots,x_{n+1})=\sum_{i=1}^{n+1}f(x_i)\prod_{1\leq j\leq n+1,\;j\neq i}\frac{1}{x_i-x_j}.\] 
The reader is referred to \cite[Lemma 2.2.4]{H2010} and \cite[Chapter 6.1.14]{HJ1991} for the proof of these properties. 

\begin{lemma}\label{extension} Let $f\in C[0,r]$ with $f(0)=0$ and assume that the restriction of $f$ to $(0,r]$ is in $C^n((0,r])$.  
Then $f^{[n]}\in C((0,r]^{n+1})$ extends to a continuous function $F\in C((0,r]^n\times [0,r])$, and 
$F(x_1,x_2,\cdots,x_n,0)=(x^{-1}f)^{[n-1]}(x_1,x_2,\cdots,x_n)$ holds. 
\end{lemma}

\begin{proof} The existence of the continuous extension $F$ follows from induction on $n$ by using the recursion   
\[f^{[n]}(x_1,x_2,\cdots,x_{n+1})=\frac{f^{[n-1]}(x_1,\cdots, x_n)-f^{[n-1]}(x_2,\cdots,x_{n+1})}{x_1-x_{n+1}} \]
because we may assume that $x_{n+1}$ is sufficiently close to 0 and $x_{n+1}\neq x_1$ to demonstrate the desired continuous extension. 

Once we know the existence of a continuous extension $F$, we may assume that $\{x_i\}_{i=1}^n$ are distinct to evaluate 
$F(x_1,\cdots, x_n,0)$. 
Thus 
\begin{align*}
\lefteqn{F(x_1,\cdots, x_n,0)=\lim_{x_{n+1}\to 0}f^{[n]}(x_1,x_2,\cdots, x_{n+1})} \\
 &=\lim_{x_{n+1}\to 0}\sum_{i=1}^{n+1}f(x_i)\prod_{1\leq j\leq n+1,\;j\neq i}\frac{1}{x_i-x_j} 
 =\sum_{i=1}^n x_i^{-1}f(x_i)\prod_{1\leq j\leq n,\;j\neq i}\frac{1}{x_i-x_j}\\
& =(x^{-1}f)^{[n-1]}(x_1,x_2,\cdots,x_n).
\end{align*}
\end{proof}

\subsection{Interpolation polynomials}\label{interpolation}
Let $A\in \M_n$ be a self-adjoint matrix with eigenvalues $b\geq \lambda_1\geq \lambda_2\geq \cdots \geq \lambda_n\geq a$. 
Then the Newton form of the Hermite interpolation polynomial for $f\in C^{n-1}[a,b]$ at the eigenvalues of $A$ is defined by 
\[p(t)=f(\lambda_1)+\sum_{k=1}^{n-1}f^{[k]}(\lambda_1,\cdots,\lambda_{k+1})\prod_{j=1}^k(t-\lambda_j).\]
It is known that $f(\lambda_i)=p(\lambda_i)$ holds for every $1\leq i\leq n$, and $f(A)=p(A)$ holds. 
Moreover, if $\lambda_i$ has a multiplicity $m$, we have $p^{(j)}(\lambda_i)=f^{(j)}(\lambda_i)$ for $0\leq j\leq m-1$  
(see \cite[Remark 1.6]{Hig2008}, \cite[Chapter 6.1.14]{HJ1991}). 
Since the eigenvalues are continuous functions of $A$, so are the coefficients of $p(x)$. 
The existence of a polynomial with such properties is used in a crucial way in the proof of Theorem \ref{Cn-1}.     
For concrete computation, we show that the coefficients can be expressed as manifestly symmetric functions of the eigenvalues, as follows:

\begin{proposition}\label{symmetric formula}
Let $A\in \M_n$ be a self-adjoint matrix with eigenvalues $b\geq \lambda_1\geq \lambda_2\geq \cdots \geq \lambda_n\geq a$,  
and let $f\in C^{n-1}[a,b]$. 
We set 
\[p_f(t)=\sum_{k=0}^{n-1}\sum_{j=0}^{n-1-k}(-1)^{n-1-j-k}e_{n-1-j-k}(\lambda_1,\cdots,\lambda_n)
(x^jf)^{[n-1]}(\lambda_1,\cdots,\lambda_n)t^k,\]
where $e_k$ is the $k$-th elementary symmetric polynomial with $n$ variables.  
Then $p_f$ coincides with the Hermite interpolation polynomial for $f\in C^{n-1}[a,b]$ at the eigenvalues of $A$.  
\end{proposition}

\begin{proof} We first show that $f(A)=p_f(A)$ holds for every polynomial $f$. 
Indeed, for $f=1$, the term $(x^jf)^{[n-1]}(\lambda_1,\cdots,\lambda_n)$ survives only if $j=n-1$, which forces $k=0$. 
Thus we get $p_1=1$, and $f(A)=p_1(A)$ holds. 
Assume that the equality holds for a polynomial $f$.  
Then thanks to the Cayley-Hamilton theorem,  
\begin{align*}
Af(A) &=\sum_{k=0}^{n-1}\sum_{j=0}^{n-1-k}(-1)^{n-1-j-k}e_{n-1-j-k}(\lambda_1,\cdots,\lambda_n)(x^jf)^{[n-1]}(\lambda_1,\cdots,\lambda_n)A^{k+1} \\
 &=\sum_{k=0}^{n-1}\sum_{j=0}^{n-k}(-1)^{n-j-k}e_{n-j-k}(\lambda_1,\cdots,\lambda_n)(x^jf)^{[n-1]}(\lambda_1,\cdots,\lambda_n)A^{k} \\
 &+f^{[n-1]}(\lambda_1,\cdots,\lambda_n)A^{n}- \sum_{j=0}^{n}(-1)^{n-j}e_{n-j}(\lambda_1,\cdots,\lambda_n)(x^jf)^{[n-1]}(\lambda_1,\cdots,\lambda_n)\\
 &=\sum_{k=0}^{n-1}\sum_{j=0}^{n-1-k}(-1)^{n-1-j-k}e_{n-1-j-k}(\lambda_1,\cdots,\lambda_n)(x^{j+1}f)^{[n-1]}(\lambda_1,\cdots,\lambda_n)A^{k} \\
 &+f^{[n-1]}(\lambda_1,\cdots,\lambda_n)\sum_{k=0}^n(-1)^{n-k}e_{n-k}(\lambda_1,\cdots,\lambda_n)A^{k}-(\chi_Af)^{[n-1]}(\lambda_1,\cdots,\lambda_n)\\
 &=p_{xf}(A),
\end{align*}
where $\chi_A$ is the characteristic polynomial of $A$. 
Thus the equality holds for any monomial by induction, and for any polynomial by linearity. 

For general $f\in C^{n-1}[a,b]$, we can choose a sequence of polynomials $\{q_l\}_l$ such that $q_l^{(k)}$ converges to $f^{(k)}$ 
uniformly on $[a,b]$ for all $0\leq k\leq n-1$. 
Then we get 
\[f(A)=\lim_{l\to\infty}q_l(A)=\lim_{l\to\infty}p_{q_l}(A)=p_f(A).\]

When $\{\lambda_i\}_{i=1}^n$ are distinct, a polynomial $q$ with $f(A)=q(A)$ whose degree is at most $n-1$ is unique, 
and $p_f$ coincides with the Hermite (in fact, Newton) interpolation polynomial $p$. 
Since the coefficients of both polynomials are continuous functions of $\{\lambda_i\}_{i=1}^n$, 
the two polynomials coincide in the general case. 
\end{proof}



\begin{thebibliography}{99}



\bibitem{AP2016}
Aleksandrov, A. B.; Peller, V. V., 
\textit{Operator Lipschitz functions.} 
Uspekhi Mat. Nauk \textbf{71} (2016), no. 4(430), 3--106; 
translation in Russian Math. Surveys \textbf{71} (2016), no. 4, 605--702. 


\bibitem{A1990} Aluthge, Ariyadasa, 
On p-hyponormal operators for $0<p<1$. 
\textit{Integral Equations Operator Theory}, 
\textbf{13} (1990), no. 3, 307--315.




\bibitem{AH1994} Ando, Tsuyoshi; Hiai, Fumio, 
Log majorization and complementary Golden-Thompson type inequalities. 
\textit{Second Conference of the International Linear Algebra Society (ILAS) (Lisbon, 1992)}, 
Linear Algebra Appl. 197/198 (1994), 113--131.


\bibitem{APS2007} Antezana, Jorge; Pujals, Enrique R.; Stojanoff, Demetrio,  
Convergence of the iterated Aluthge transform sequence for diagonalizable matrices. 
\textit{Adv. Math.} \textbf{216} (2007), no. 1, 255--278.

\bibitem{APS2011} Antezana, Jorge; Pujals, Enrique R.; Stojanoff, Demetrio,  
The iterated Aluthge transforms of a matrix converge. 
\textit{Adv. Math.} \textbf{226} (2011), no. 2, 1591--1620.


\bibitem{B1997} Bhatia, Rajendra. 
\textit{Matrix analysis.}
Grad. Texts in Math., 169, Springer-Verlag, New York, 1997.



\bibitem{Br1994}
Bruckner, Andrew, 
\textit{Differentiation of real functions.
Second edition}.  
CRM Monogr. Ser., 5
American Mathematical Society, Providence, RI, 1994. 



\bibitem{CJL2005} Cho, Muneo; Jung, Il Bong; Lee, Woo Young, 
On Aluthge transforms of $p$-hyponormal operators. 
\textit{Integral Equations Operator Theory}, \textbf{53} (2005), no. 3, 321--329.

\bibitem{DK1974}
Dalec'kii, Ju. L.; Krein, M. G. 
\textit{Stability of solutions of differential equations in Banach space.} 
Translated from the Russian by S. Smith
Transl. Math. Monogr., Vol. \textbf{43}, 
American Mathematical Society, Providence, RI, 1974.






\bibitem{DS2009} Dykema, Ken; Schultz, Hanne, 
Brown measure and iterates of the Aluthge transform for some operators arising from measurable actions. 
\textit{Trans. Amer. Math. Soc.} \textbf{361} (2009), no. 12, 6583--6593.



\bibitem{GK1969} Gohberg, I. C.; Krein, M. G. 
\textit{Introduction to the theory of linear nonselfadjoint operators.} 
Translated from the Russian by A. Feinstein, 
Transl. Math. Monogr., Vol. 18, 
American Mathematical Society, Providence, RI, 1969. 

\bibitem{HS2009} Haagerup, Uffe; Schultz, Hanne, 
Invariant subspaces for operators in a general II$_1$-factor. 
\textit{Publ. Math. Inst. Hautes \'Etudes Sci.} No. \textbf{109} (2009), 19--111.

\bibitem{Hal1982}
Halmos, Paul Richard, 
\textit{A Hilbert space problem book,} 
Second edition, Encyclopedia Math. Appl., 17
Grad. Texts in Math., 19
Springer-Verlag, New York-Berlin, 1982. 


\bibitem{de la Harpe1972} de la Harpe, Pierre, 
\textit{Classical Banach-Lie algebras and Banach-Lie groups of operators in Hilbert space.} 
Lecture Notes in Math., Vol. 285 Springer-Verlag, Berlin-New York, 1972.

\bibitem{H2010} Hiai, Fumio, 
Matrix analysis: matrix monotone functions, matrix means, and majorization. 
\textit{Interdiscip. Inform. Sci.} \textbf{16} (2010), no. 2, 139--248.

\bibitem{HP2014} Hiai, Fumio; Petz, D\'enes, 
\textit{Introduction to matrix analysis and applications.} 
Universitext Springer, Cham; Hindustan Book Agency, New Delhi, 2014. 

\bibitem{Hig2008}
Higham, Nicholas J., 
\textit{Functions of matrices.  
Theory and computation.}
Society for Industrial and Applied Mathematics (SIAM), Philadelphia, PA, 2008. 

\bibitem{HJ1991} 
Horn, Roger A.; Johnson, Charles R.
\textit{Topics in matrix analysis}.
Cambridge University Press, Cambridge, 1991. 

\bibitem{K1973} 
Kato, Tosio, 
Continuity of the map $S\mapsto |S|$ for linear operators.
\textit{Proc. Japan Acad.} \textbf{49} (1973), 157--160.



\bibitem{K1992} Kosaki, Hideki, 
Unitarily invariant norms under which the map $A\to |A|$ is Lipschitz continuous. 
\textit{Publ. Res. Inst. Math. Sci.} \textbf{28} (1992), no. 2, 299--313.

\bibitem{OY2025}
Osaka, Hiroyuki; Yamazaki, Takeaki, 
Limit of iteration of the induced Aluthge transformations of centered operators. 
\textit{Trans. Amer. Math. Soc.} \textbf{378} (2025), no. 10, 6857--6884.

\bibitem{PS2011} Potapov, Denis; Sukochev, Fedor, 
Operator-Lipschitz functions in Schatten-von Neumann classes. 
\textit{Acta Math.} \textbf{207} (2011), no. 2, 375--389. 


\bibitem{S2005} Simon, Barry. 
\textit{Trace ideals and their applications.} 
Second edition, Math. Surveys Monogr., 120, 
American Mathematical Society, Providence, RI, 2005.

\bibitem{Ta2008}
Tam, Tin-Yau, 
$\lambda$-Aluthge iteration and spectral radius. 
\textit{Integral Equations Operator Theory}, \textbf{60} (2008), no. 4, 591--596.

\bibitem{Te2012} Teschl, Gerald, 
\textit{Ordinary differential equations and dynamical systems.} 
Grad. Stud. Math., 140
American Mathematical Society, Providence, RI, 2012. 


\bibitem{Y2002} Yamazaki, Takeaki, 
\textit{An expression of spectral radius via Aluthge transformation.}
\textit{Proc. Amer. Math. Soc.} \textbf{130} (2002), 1131--1137. 
\end{thebibliography}
\end{document}